\documentclass[a4paper]{amsart}

\usepackage{amsmath,amssymb,amsthm,amsfonts}
\usepackage[UKenglish]{babel}
\usepackage{microtype}
\usepackage[shortlabels]{enumitem}

\usepackage[bookmarks=true,hyperindex,pdftex,colorlinks,citecolor=blue]{hyperref}

\usepackage{bbm} 

\DeclareMathOperator{\supp}{supp}                           
\DeclareMathOperator{\Lip}{Lip}                             

\newcommand{\N}{\mathbb{N}}             
\newcommand{\R}{\mathbb{R}}             
\newcommand{\J}{\mathcal{J}}            

\newcommand{\car}{\mathbbm{1}}   

\newcommand{\ep}{\varepsilon} 

\newcommand{\abs}[1]{\lvert{#1}\rvert}                     
\newcommand{\set}[1]{\left\{{#1}\right\}}                   
\newcommand{\norm}[1]{\lVert{#1}\rVert}                  

\newcommand{\F}{\mathcal{F}}                             

\def\<{\langle}
\def\>{\rangle}
\newcommand{\eqnorm}[1]{{\left\vert\kern-0.25ex\left\vert\kern-0.25ex\left\vert #1 
    \right\vert\kern-0.25ex\right\vert\kern-0.25ex\right\vert}}
\newcommand\eqnorma{\eqnorm{\cdot}}

\def\Sz{\operatorname{Sz}}
\def\d{\overline{\delta}}
\def\p{\overline{\rho}}

\newtheorem{theorem}{Theorem}[section]
\newtheorem{lemma}[theorem]{Lemma}

\newtheorem{proposition}[theorem]{Proposition}
\newtheorem{corollary}[theorem]{Corollary}

\newtheorem*{T1}{Theorem~\ref{TheoremWeakStarMP}} 

\newtheorem*{T2}{Theorem~\ref{TheoremCompactWMP}} 

\newtheorem*{T3}{Theorem~\ref{thm:renorming}} 

\theoremstyle{definition}
\newtheorem{definition}[theorem]{Definition}
\newtheorem{example}[theorem]{Example}
\newtheorem{question}[theorem]{Question}
\theoremstyle{remark}
\newtheorem{remark}[theorem]{Remark}

\begin{document}
	
	\title[On the Weak Maximizing Properties]{On the Weak Maximizing Properties}
	
	\author[L. Garc\'ia-Lirola]{Luis C. Garc\'ia-Lirola}
	\address{Departamento de Matem\'aticas, Instituto Universitario de Matemáticas y Aplicaciones, Universidad de Zaragoza, 50009 Zaragoza, Spain}
	\email{luiscarlos@unizar.es}
	
	\author[C. Petitjean]{Colin Petitjean}
	\address{LAMA, Univ Gustave Eiffel, UPEM, Univ Paris Est Creteil, CNRS, F-77447, Marne-la-Vall\'ee, France}
	\email{colin.petitjean@univ-eiffel.fr}
	
	\date{November, 2020} 
	
	\begin{abstract}
		Quite recently, a new property related to norm-attaining operators has been introduced: the weak maximizing property (WMP). In this note, we define a generalised version of it considering other topologies than the weak one (mainly the weak$^*$ topology). We
		provide new sufficient conditions, based on the moduli of asymptotic uniform smoothness and convexity, which imply that a pair $(X,Y)$ enjoys a certain maximizing property. This approach not only allows us to (re)obtain as a direct consequence that the pair $(\ell_p,\ell_q)$ has the WMP but also provides many more natural examples of pairs having a given maximizing property. 
	\end{abstract}

	
	\subjclass{Primary 46B20; Secondary 46B04, 54E50}

	\keywords{Norm attainment, Banach Space, Weak Maximizing Property, Asymptotic Uniform Smoothness, Asymptotic Uniform Convexity}

	\maketitle

	\section{Introduction}
	Let $X,Y$ be two real Banach spaces and let $T\colon X \to Y$ be a bounded linear operator (we will write $T \in \mathcal{L}(X,Y)$). A maximizing sequence for $T$ is a sequence $(x_n)_n \subset X$ with $\|x_n\|=1$ for every $n \in \N$ and such that $\lim\limits_{n \to \infty} \|Tx_n\|_Y = \|T\|$. Next, we say as usual that $T\colon X \to Y$ attains its norm whenever there exists a vector $x \in X$ of norm 1 such that $\|Tx\|_Y = \|T\|$. 
	Now following \cite{Aron_WMP19}, a pair of Banach spaces $(X,Y)$ is said to have the weak maximizing property (WMP) if for any bounded linear operator $T\colon X \to Y$, the existence of a non-weakly null maximizing sequence for $T$ implies that $T$ attains its norm.
	For instance, it is proved in \cite[Theorem~1]{Pellegrino_09} that the pair $(\ell_p,\ell_q)$ has the WMP whenever $1<p<\infty$ and $1 \leq q < \infty$. That result was extended in \cite[Proposition 2.2]{Aron_WMP19} to the pair $(\ell_p(\Gamma_1),\ell_q(\Gamma_2))$ where $\Gamma_1,\Gamma_2$ are arbitrary index sets.

	Needless to say that the theory of norm attaining operators finds many applications in both pure and applied mathematics. A motivation for the study of the WMP is the following application which was noted in \cite{Aron_WMP19} (and extends a former result due to J.~Kover in the Hilbert case \cite{Kover2005}): if $(X,Y)$ has the weak maximizing property, $T \in \mathcal{L}(X,Y)$ and $K\colon X \to Y$ is a compact operator such that $\|T\| < \|T+K\|$, then $T+K$ is norm attaining. 
	As a consequence, the authors could deduce that a pair $(X,Y)$ has the weak maximizing property for some $Y \neq \set{0}$ if and only if $X$ is reflexive. 
	
	In this note, we provide a new (possibly simpler) approach based on a comparison between the modulus of asymptotic uniform convexity of $X$ and the modulus of asymptotic uniform smoothness of $Y$, and conversely; see Section~\ref{section_AUSAUC} for precise definitions.  Moreover, we consider the following related and quite natural properties.
	
	\begin{definition} 
		Let $X$ and $Y$ be Banach spaces. Let $\tau_X$ and $\tau_Y$ be any topology on $X$ and $Y$ respectively.
		\begin{itemize}[leftmargin=*]
			\item We say that a pair $(X,Y)$ has the $\tau_X$-to-$\tau_Y$ maximizing property ($\tau_X$-to-$\tau_Y$MP) if for any $T \in  \mathcal{L}(X,Y)$ which is $\tau_X$-to-$\tau_Y$ continuous, the existence of a non $\tau_X$-null maximizing sequence for $T$ implies that $T$ attains its norm. 
		\end{itemize}
		
		In this paper, $\tau_X$ and $\tau_Y$ will mainly be the usual weak or weak$^*$ topologies. Since any $T \in  \mathcal{L}(X,Y)$ is weak-to-weak continuous, it is clear that the WMP corresponds to the weak-to-weakMP. However, there are operators $T \in  \mathcal{L}(X^*,Y^*)$ which are not weak$^*$-to-weak$^*$ continuous, so we will also consider the next property:
		
		\begin{itemize}[leftmargin=*]
			\item We say that a pair $(X^*,Y)$ has the weak$^*$ maximizing property (W$^*$MP) if for any $T \in  \mathcal{L}(X^*,Y)$, the existence of a non-weak$^*$ null maximizing sequence for $T$ implies that $T$ attains its norm. 
		\end{itemize}
		Notice that if the pair $(X^*,Y)$ has the WMP then it also has the W$^*$MP, and consequently the weak$^*$-to-weakMP. Also, if $X$ is reflexive, then a pair $(X,Y)$ has the WMP if and only if $(X,Y)$ has the W$^*$MP if and only if $(X,Y)$ has the weak$^*$-to-weakMP. However, we will prove that these three properties do not coincide in general.  
	\end{definition}

	\begin{remark}
		The definition of the above properties can be stated in terms of nets instead of sequences. Indeed, given an operator $T\colon X\to Y$, the following properties are equivalent:
		\begin{itemize}
			\item[i)] There exists a non $\tau_X$-null maximizing sequence for $T$. 
			\item[ii)] There exists a non $\tau_X$-null maximizing net for $T$.
			\item[iii)] There exists a net which is maximizing for $T$ and does not admit $0$ as a $\tau_X$-cluster point.
		\end{itemize}
		Clearly, i)$\Rightarrow$ii). 
		To see that iii)$\Rightarrow$ i), let  $(x_\alpha)_\alpha\subset S_{X}$ be a net such that $0$ is not a $\tau_X$-cluster point and $\lim_\alpha\norm{Tx_\alpha}=\norm{T}$. Pick inductively $\alpha_n$ such that $\norm{Tx_{\alpha_n}}\geq 1-\frac{1}{n}$ and $\alpha_n\geq \alpha_m$ if $n\geq m$. Then the sequence $(x_{\alpha_n})_n$ is not $\tau_X$-null, and it is maximizing for $T$. 
		Finally, assume that ii) holds and let's prove iii). Let $(x_\alpha)_\alpha$ be a non $\tau_X$-null net maximizing for $T$. If $0$ is the only $\tau_X$-cluster point of $(x_\alpha)_\alpha$, then all the subnets of $(x_\alpha)_\alpha$ converge to $0$, which is a contradiction (see e.g. \cite{AliprantisBorder}). 
		Thus, $0$ is not the only $\tau_X$-cluster point of $(x_\alpha)_\alpha$, that is, there is a subnet $(y_\beta)_\beta$ convergent to $y\neq 0$, and $(y_\beta)_\beta$ is also maximizing for \nolinebreak$T$.  
	\end{remark}
	\begin{remark}
		Replacing the WMP by the weak$^*$-to-weak$^*$MP, one can follow the lines of \cite[Proposition 2.4]{Aron_WMP19} to obtain the following result:\\
		Suppose that $(X^*,Y^*)$ has the weak$^*$-to-weak$^*$MP. Let $T,K \colon X^* \to Y^*$ be weak$^*$-to-weak$^*$ continuous linear operators such that $K$ is compact. 
		If $\norm{T} < \norm{T+K}$ then $T+K$ is norm attaining.
	\end{remark}
	We now describe the main findings of this paper.  Throughout the paper, $X$ and $Y$ will denote real Banach spaces while $X^*$ denotes, as usual, the topological dual of $X$.  
	After recalling the definition of the modulus $\d_{X^*}^*(t)$ of weak$^*$ asymptotic uniform convexity and of the modulus $\p_Y(t)$ of asymptotic uniform smoothness in Section~\ref{section_AUSAUC}, we prove our first theorem in Section~\ref{section-mainresults}.
	\begin{T1} \label{T1}
		Let $X,Y$ be Banach spaces. Assume that for every $t>0$, $\overline{\delta}_{X^*}^*(t)\geq \p_Y(t)$ and, for all $t\geq 1$,
		$\overline{\delta}_{X^*}^*(t)>t-1$. Then,
		\begin{enumerate}
			\item the pair $(X^*,Y)$ has the weak$^*$-to-weakMP.
			\item If moreover $Y\equiv Z^*$ is a dual space, then the pair $(X^*,Z^*)$ has the weak$^*$-to-weak$^*$MP.
			\item If $\d_{X^*}^*(t)=t$, then $(X^*,Z)$ has the W$^*$MP for any Banach space $Z$.
		\end{enumerate} 
		In particular, if $X$ is reflexive then the pair $(X^*,Y)$ has the WMP.
	\end{T1}
	Since for any infinite sets $\Gamma_1$ and $\Gamma_2$ and for $1<p<q<\infty$, it is well known that  $\d_{\ell_p(\Gamma_1)}^*(t) = (1+t^p)^{1/p} - 1 > (1+t^q)^{1/q} - 1 = \p_{\ell_q(\Gamma_2)}$, our last theorem provides a new proof of the fact that the pair $(\ell_p(\Gamma_1), \ell_q(\Gamma_2))$ has the WMP. Note that the case $p=q$ also follows from the last theorem.
	In both papers \cite{Aron_WMP19,Pellegrino_09}, the proof for the case $1\leq q<p<\infty$
	follows from Pitt's theorem asserting that any bounded operator $T\colon \ell_p \to \ell_q$ is compact and thus attains its norm. In fact, Pitt's theorem can also be generalised using the modulus $\p_X(t)$ of asymptotic uniform smoothness of $X$ and the modulus $\d_Y(t)$ of asymptotic uniform convexity of $Y$. Namely, if there exists $t>0$ such that $\p_X(t) < \d_Y(t)$, then every bounded linear operator from $X$ to $Y$ is compact (see \cite[Proposition~2.3]{JLPS02}, where this is stated only for $0<t<1$, but the
	proof works for any $t >0$). As an easy consequence we deduce the next result. 
	\begin{T2}
		Let $X,Y$ be Banach spaces and $\eqnorma$ be an equivalent norm on $X^*$. If there exists $t>0$ such that $\p_{\eqnorma}(t) < \d_Y(t)$, then
		\begin{enumerate}
			\item the pair $(X^*,Y)$ has the weak$^*$-to-weakMP.
			\item If moreover $Y \equiv Z^*$ is a dual space, then $(X^*,Z^*)$ has the weak$^*$-to-weak$^*$MP.
		\end{enumerate}
		In particular, if $X$ is reflexive then $(X^*,Y)$ has the WMP. 
	\end{T2}

	In Section~\ref{section-renorming}, we take advantage of the renorming theory to enlarge the range of applications of Theorem~\ref{TheoremWeakStarMP}. Notably, we deduce that if $X$ belongs to a special class of Banach spaces, one can find an equivalent norm $\eqnorma$ on $X^*$ as well as some $q \geq 1$ such that the pair $\big((X^*,\eqnorma),\ell_q\big)$ has the weak$^*$-to-weak$^*$MP. In particular, if $X$ is moreover reflexive then $\big((X^*,\eqnorma),\ell_q\big)$ has the WMP. Furthermore, there is a strong relationship between asymptotic moduli and lower/upper estimates for spaces having finite dimensional decompositions (shortened FDDs). This allows us to prove the next theorem.
	\begin{T3}
		Let $X, Y$ be Banach spaces with shrinking FDDs. Assume that the norm of $X^*$ is $p$-AUC* and the norm of $Y$ is $q$-AUS for some $1<p \leq q<\infty$. Then there exist equivalent norms $\eqnorma_X$ on $X$ and $\eqnorma_Y$ on $Y$ such that the pair $\big((X^*,\eqnorma_{X}^*), (Y,\eqnorma_Y)\big)$ has the weak*-to-weakMP.
	\end{T3}
	
	Finally, in Section~\ref{section-applications}, we deal with some classical Banach spaces.
	For instance, it is readily seen that Schur spaces are the best range spaces for the WMP: the pair $(X,Y)$ has the WMP for any reflexive space $X$ and any Schur space $Y$ (Proposition~\ref{prop-Schur-range}). Also, thanks to Theorem~\ref{TheoremWeakStarMP}~(3) and $\d_{\ell_1}
	^*(t) = t$, the space $\ell_1=c_0^*$ is also a very good domain for the W$^*$MP since 
	$(\ell_1,Y)$ has the  W$^*$MP for every Banach space $Y$ (see Corollary~\ref{Cor-l1-domain}). The latter result does not hold for every Schur space as it is shown for instance by Example~\ref{Example1-Schur}. To finish, we discuss the case of Dunford--Pettis spaces, James sequence spaces, Orlicz spaces and we also add some comments about the pair $(L_p([0,1]),L_q([0,1]))$.

	\section{Preliminaries: Asymptotic uniform smoothness and convexity}
	\label{section_AUSAUC}
	
	Consider a real Banach space $(X,\norm{\cdot})$ and let $S_X$ be its unit sphere. 
	The \emph{modulus of asymptotic uniform convexity} of $X$ is given by
	\[ \forall t >0, \quad \d_X(t) = \inf_{x\in S_X} \sup_{\dim(X/Y)<\infty}\inf_{y\in S_Y} \norm{x+ty}-1\,. \]
	The space $(X,\norm{\cdot})$ is said to be \emph{asymptotically uniformly convex} (AUC for short) if $\d_X(t)>0$ for each $t>0$.  When $X$ is a dual space and $Y$ runs through all finite-codimensional weak$^*$-closed subspaces of $X$, the corresponding modulus is denoted by $\d_X^*(t)$. 
	Then the space $X$ is said to be \emph{weak$^*$ asymptotically uniformly convex} (AUC*) if $\d_X^*(t)>0$ for each $t>0$.
	The \emph{modulus of asymptotic uniform smoothness} of $X$ is given by
	\[ \forall t >0, \quad \p_X(t) = \sup_{x\in S_X} \inf_{\dim(X/Y)<\infty}\sup_{y\in S_Y} \norm{x+ty}-1\,. \]
	The space $(X,\norm{\cdot})$ is said to be \emph{asymptotically uniformly smooth} (AUS for short) if $\lim_{t\to 0} t^{-1}\p_X(t) = 0$.
	We refer the reader to~\cite{JLPS02} and the references therein for a detailed study of these properties.
	\medskip
	
	Let $p,q\in [1,\infty)$. We say that  $X$ is  \emph{weak$^*$ $p$-asymptotic uniformly convex} (abbreviated by $p$-AUC$^*$) if there exists $C>0$ so that $\overline{\delta}_X^*(t)\geq Ct^p$ for all $t\in [0,1]$. Similarly, we say that $X$ is $q$-asymptotic uniformly smooth (abbreviated by $q$-AUS) if there exists $C >0$ so that $\p_X(t) \leq C t^q$ for all $t \in [0,1]$.
	Let us highlight that the following is proved in  \cite[Corollary 2.4]{DKLR}.
	
	\begin{proposition}\label{duality} Let $X$ be a Banach space.
		\begin{enumerate}[(i)]
			\item Then $\norm{\cdot}_X$ is AUS if and and only if $\norm{\cdot}_{X^*}$ is AUC$^*$.
			\item If $p\in (1,\infty]$ and $q\in [1,\infty)$ are conjugate exponents, then $\norm{\cdot}_X$ is $p$-AUS if and only if $\norm{\cdot}_{X^*}$ is $q$-AUC$^*$.
		\end{enumerate}
	\end{proposition}
	
	It is worth mentioning that if $X\equiv Z^*$ and the dual norm $\norm{\cdot}_{Z^*}$ is AUS, then $X$ must be reflexive (see for instance Proposition 2.6 in \cite{CauseyLancien}). The following proposition is elementary.
	
	\begin{proposition}\label{as-sequences} Let $X$ be a Banach space. For any weakly null net $(x_\alpha)_\alpha$ in $X$ 
		and any  $x \in X\setminus\{0\}$ we have:
		\[
		\limsup_{\alpha} \norm{x+x_\alpha} \leq \norm{x}\left(1+\overline{\rho}_X\left(\frac{\limsup_\alpha \norm{x_\alpha}}{\norm{x}}\right)\right).
		\]
		
		For any weak$^*$-null net $(x^*_\alpha)_\alpha \subset X^*$ and for any $x^* \in X^*\setminus \set{0}$ we have
		\[
		\liminf_{\alpha} \norm{x^*+x^*_\alpha} \geq \norm{x^*}\left(1+\overline{\delta}_X^*\left(\frac{\liminf_\alpha \norm{x_\alpha^*}}{\norm{x^*}}\right)\right).
		\]
	\end{proposition}
	
	Assume now that $\varphi\colon [0,\infty) \to [0,\infty)$ is a 1-Lipschitz convex function with  $\lim_{t\to \infty}\varphi(t)/t=1$ and $\varphi(t)\geq t-1$ for all $t\geq 0$. 
	Consider for $(s,t) \in \R^2$, 
	\[
	N_2^\varphi(s,t)=
	\begin{cases}\abs{s}+\abs{s}\varphi(\abs{t}/\abs{s}) & \text{ if }s\neq 0,\\
	\abs{t}& \text{ if }s=0.
	\end{cases}
	\]
	The following is stated  
	in \cite{KaltonTAMS2013} (see Lemma~4.3 and its preparation), we include a proof for completeness.

	\begin{lemma}\label{Orlicz-Kalton} The function $N_2^\varphi$ is an absolute (or lattice) norm on $\R^2$, meaning that $N_2^\varphi(s_1,s_2)\le N_2^\varphi(t_1,t_2)$, whenever $|s_i|\le |t_i|$ for all $i\le 2$.
	\end{lemma}

	\begin{proof} First, note that if $0<u<v$ then 
		\[ \frac{\varphi(v)-\varphi(u)}{v-u} \leq 1\leq \frac{1+\varphi(v)}{v}\]
		since $\varphi$ is $1$-Lipschitz and $\varphi(v)\geq v-1$. It follows that  $\frac{1}{v}(1+\varphi(v))\leq \frac{1}{u}(1+\varphi(u))$. Thus, 
		\[ N_2^\varphi(s,1)=s(1+\varphi(1/s))\leq s'(1+\varphi(1/s')) = N_2^\varphi(s',1) \quad \text{for all } 0<s<s'.\]
		From that and the positive homogeneity of $N_2^\varphi$, its clear that 
		\begin{equation}\label{eq:increasing}
		N_2^\varphi(s,t)\leq N_2^\varphi(s',t') \quad  \text{for all } |s|\leq |s'| \text{ and } |t|\leq |t'|. 
		\end{equation}
		Now, assume that $0<s,s', t,t'$. The convexity of $\varphi$ yields
		\begin{align*}
		N_2^\varphi(s+s',t+t')&=(s+s')\left(1+\varphi\left(\frac{t+t'}{s+s'}\right)\right) \\
		&= s+s'+ (s+s')\varphi\left(\frac{s}{s+s'}\frac{t}{s}+\frac{s'}{s+s'}\frac{t'}{s'}\right) \\
		&\leq s+s'+s\varphi(t/s)+s'\varphi(t'/s') = N_2^\varphi(s,t)+N_2^\varphi(s',t'). 
		\end{align*}
		From that and \eqref{eq:increasing}, we get
		\begin{align*} N_2^\varphi(s+s',t+t')&=N_2^\varphi(|s+s'|,|t+t'|)\leq N_2^\varphi(|s|+|s'|, |t|+|t'|)\\
		&\leq N_2^\varphi(|s|, |t|)+N_2^\varphi(|s'|,|t'|)
		=N_2^\varphi(s,t)+N_2^\varphi(s',t')
		\end{align*}
		whenever $s,s'\neq 0$. The case where $s=0$ or $s'=0$ follows from the continuity of $N_2^\varphi$. This shows that $N_2^\varphi$ defines a norm on $\mathbb R^2$, and clearly it is absolute. 
	\end{proof}
	
	\begin{remark}
		In the proof of \cite[Proposition~4.1]{KaltonCompact}, it is assumed that $\varphi$ is such that $t \in (0,\infty) \mapsto t^{-1}(\varphi(t)+1)$ is monotone decreasing, which is obtained in the previous proof under the stronger hypothesis that $\varphi(t)\geq t-1$. For our purposes, there is no loss of generality. Indeed, when $X$ is a Banach space, it is easy to see that $\overline{\rho}_X$ is a 1-Lipschitz convex function such that $\lim_{t\to \infty}\overline{\rho}_X(t)/t=1$. Moreover, $\overline{\rho}_X(t)\geq \overline{\rho}_{c_0}(t)=\max\{0,t-1\}$ so that one can consider the norm $N_2^{\overline{\rho}_X}$.
	\end{remark}
	
	We will use in the sequel the following reformulation of Proposition~\ref{as-sequences} in terms of the norm $N_2^{\overline{\rho}_X}$.
	\begin{lemma}\label{l:asymptotic-Nnorm}
		Let $X$ be a Banach space.  If $(x_\alpha)_\alpha \subset X$ is weakly-null and $x\in X$, then 
		\[ \limsup_{\alpha} \| x+x_\alpha \| \leq N_2^{\overline{\rho}_X}(\norm{x},\limsup_\alpha \norm{x_\alpha}).\]
	\end{lemma}
	
	\begin{proof}
		If $x=0$ there is nothing to do, so we may assume that $x\neq 0$.
		By application of Proposition~\ref{as-sequences} we see that 
		\[
		\begin{split}
		\limsup_{\alpha}\norm{x+x_\alpha}&\leq \norm{x}\left(1+\overline{\rho}_X\left(\frac{\limsup_\alpha\norm{x_\alpha}}{\norm{x}}\right)\right) \\
		&=N_2^{\overline{\rho}_X}(\norm{x},\limsup_\alpha \norm{x_\alpha}).
		\end{split}
		\]
	\end{proof}
	
	\begin{example}\label{ex:ausauc} Let us recall the asymptotic moduli for some classical spaces, the three first examples are mostly taken from \cite{Milman71}. 
		\begin{enumerate}[(i)]
			\item Let $1\leq p<\infty$. If $X=(\sum_{n=1}^\infty E_n)_p$, where $\dim(E_n)<\infty$, then $\d_{X}(t)=\p_{X}(t)=(1+t^p)^{1/p}-1$ for all $t>0$.  
			\item For any infinite set $\Gamma$ and $1\leq p<\infty$, $\d_{\ell_p(\Gamma)}(t)=\p_{\ell_p(\Gamma)}(t)=(1+t^p)^{1/p}-1$ for all $t>0$. 
			\item For any infinite set $\Gamma$, $\d_{c_0(\Gamma)}(t)=\p_{c_0(\Gamma)}(t)=\max\{0,t-1\}$ for all $t>0$. 
			\item Let $J$ be the James space. Then $\d_{J}(t)=(1+t^2)^{1/2}-1$ for all $t>0$, and there is an equivalent norm $|\cdot|$ on $J$ so that $\p_{|\cdot|}(t)\leq (1+t^2)^{1/2}-1$ (see Proposition~\ref{Jsmooth}).  
			\item If $X$ has a Schauder basis $(e_n)_{n=1}^\infty$ which satisfies a lower (resp. upper) $p$-estimate with constant one, then $\d_{X}(t)\geq (1+t^p)^{1/p}-1$ (resp. $\p_{X}(t)\leq (1+t^p)^{1/p}-1$). For instance, the Lorentz sequence space $d(w,p)$ satisfy an upper $p$-estimate with constant 1 \cite[p.~177]{LT77}. Moreover, $d(w,p)$ contains almost isometric copies of $\ell_p$ so $\p_{d(w,p)}(t)=(1+t^p)^{1/p}-1$ for all $t>0$.
			\item Given a Banach space $X$ with a $1$-unconditional Schauder basis $(e_n)_{n=1}^\infty$ and $1\leq p<\infty$, its $p$-convexification $X^p$ is defined as 
			\[X^p=\{(x_n)_{n=1}^\infty : \sum_{n=1}^\infty |x_n|^p e_n\in X\} \]
			with the norm $\norm{(x_n)_{n=1}^\infty}_p=\norm{\sum_{n=1}^\infty |x_n|^p e_n}^{1/p}$. It follows from the triangle inequality that $X^p$ satisfies an upper $p$-estimate with constant $1$ which readily implies that $\overline{\rho}_{X^p}(t)\leq (1+t^p)^{1/p}-1$. 
			\item  The predual $JT_*$ of the James tree space $JT$ as well as its dual $JT^*$ are asymptotically uniformly convex \cite{Girardi}. Moreover, there exists a positive constant $c$ so that $\d_{JT_*}(t) \geq ct^3$. 
			\item The Tsirelson space $\mathrm{T}$ can be renormed for every $p>1$ in such a way that $\d_\mathrm{T}(t) \geq c_p t^p$ \cite[Remarks 7.2]{KOS99}. By duality, we have that $\mathrm{T}^*$ can be renormed for every $q>1$ in such a way that $\p_{\mathrm{T}^*}(t) \leq c_q' t^q$. On the other hand, $\mathrm{T}$ does not have an equivalent AUS norm since otherwise it would follow from \cite{KOS99} (see also \cite{DKLR}) that the Szlenk index of $\mathrm{T}$ is $\omega$, which is not the case thanks to \cite{Odell2007} (see also \cite{Raja2013})).
		\end{enumerate}
	\end{example}
	
	\section{Proof of the main results} \label{section-mainresults}
	Let us begin with the proof of our main theorem.
	\begin{theorem} \label{TheoremWeakStarMP}
		Let $X,Y$ be Banach spaces. Assume that for every $t>0$, $\overline{\delta}_{X^*}^*(t)\geq \p_Y(t)$ and, for all $t\geq 1$,
		$\overline{\delta}_{X^*}^*(t)>t-1$. Then,
		\begin{enumerate}
			\item the pair $(X^*,Y)$ has the weak$^*$-to-weakMP.
			\item If moreover $Y\equiv Z^*$ is a dual space, then the pair $(X^*,Z^*)$ has the weak$^*$-to-weak$^*$MP.
			\item If $\d_{X^*}^*(t)=t$ then $(X^*,Z)$ has the W$^*$MP for any Banach space $Z$.
		\end{enumerate} 
		In particular, if $X$ is reflexive then the pair $(X^*,Y)$ has the WMP.
	\end{theorem}

	\begin{proof}
		We will only prove the first assertion since the proof of the second and third one are very similar (replacing the weak topology in $Y$ by the weak$^*$ topology in $Y$). 
		
		Let $T\colon X^* \to Y$ be a bounded operator which is weak$^*$-to-weak continuous. Without loss of generality, we may assume that $T$ has norm 1. Let $(x_n)$ be a normalized maximizing sequence in $X^*$ which is not weak$^*$-null. Let $(x_\alpha)$ be a subnet which is weak$^*$ convergent to $x \neq 0$. By extracting a subnet again, we may assume that $\lim_\alpha \norm{x_\alpha-x}$ and $\lim_\alpha \norm{Tx-Tx_\alpha}$ exists. 
		Since $T$ is weak$^*$-to-weak continuous, we have that $Tx_\alpha \overset{w}{\longrightarrow} Tx$. Using Lemma~\ref{l:asymptotic-Nnorm}, we thus have the following estimates: 
		\begin{eqnarray*}
			1 &=& \norm T = \lim\limits_{\alpha} \norm{Tx_\alpha}  = \lim\limits_{\alpha} \norm{Tx+Tx_\alpha - Tx} \\
			&\leq& N_2^{\overline{\rho}_Y}(\norm{Tx},\lim_\alpha \norm{Tx_\alpha -Tx}).
		\end{eqnarray*}
		Now notice that (we use Lemma~\ref{Orlicz-Kalton})
		\begin{eqnarray*}
			N_2^{\overline{\rho}_Y}(\norm{Tx},\lim_\alpha \norm{Tx_\alpha -Tx}) &\leq& N_2^{\overline{\rho}_Y}(\norm{Tx},\lim_\alpha \norm{x_\alpha -x}) \\
			&\leq& N_2^{\overline{\rho}_Y}(\norm{x},\lim_\alpha \norm{x_\alpha -x}).
		\end{eqnarray*}
		Since $x \neq 0$, we deduce from the definition of $N_2^{\overline{\rho}_Y}$ the following estimates:
		\begin{eqnarray*}
			N_2^{\overline{\rho}_Y}(\norm{x},\lim_\alpha \norm{x_\alpha -x}) &\leq& \|x\|  +\|x\|\overline{\rho}_Y\left( \frac{\lim_\alpha \norm{x_\alpha -x}}{\|x\|} 
			\right) \\
			&\leq& \|x\|  +\|x\|\overline{\delta}_{X^*}^*\left( \frac{\lim_\alpha \norm{x_\alpha -x}}{\|x\|}  \right) \\
			&\leq& \lim_\alpha \norm{x +x_\alpha -x} \\
			&=& 1.
		\end{eqnarray*}
		This implies that all the previous inequalities are in fact equalities, in particular
		\begin{eqnarray*} N_2^{\p_Y}(\norm{Tx},\lim_\alpha \norm{x_\alpha -x}) = N_2^{\p_Y}(\norm{x},\lim_\alpha \norm{x_\alpha -x})=1 = N_2^{\p_Y}(0,1).
		\end{eqnarray*}
		This means that the points $(\norm{Tx},\lim_\alpha \norm{x_\alpha -x})$, $(\norm{x},\lim_\alpha \norm{x_\alpha -x})$ and $(0,1)$ are aligned in $\mathbb R^2$. If $\norm{Tx}=\norm{x}$ or $\lim_\alpha \norm{x-x_\alpha}=0$, then $T$ attains its norm at $x$ and we are done. Otherwise, it follows that $\lim_\alpha\norm{x_\alpha-x}=1$ and  
		\[ 1= N_2^{\p_Y}(\norm{x},1) = \norm{x}+\norm{x}\p_Y(1/\norm{x}) = \norm{x}+\norm{x}\d_{X^*}^*(1/\norm{x}). \]
		That is, $\d_{X^*}(\frac{1}{\norm{x}})=\p_Y(\frac{1}{\norm{x}})=\frac{1}{\norm{x}}-1$, which contradicts our assumptions. 
	\end{proof}
	
	\begin{remark}\label{rem:failing} Take $X=\mathbb R\oplus_{\infty} \ell_2$ and $Y=c_0$. Then $\d_X(t)=\max\{0,t-1\}=\p_Y(t)$ for all $t>0$. However, the operator $T\colon X\to Y$ given by $T((0,e_n))=\frac{n}{n+1}e_n$ and $T((1,0))=0$ does not attain the norm and admits the non-weakly null maximizing sequence $(x_n)_{n=1}^\infty$ given by $x_n=(1,e_n)$. Thus, the pair $(X, Y)$ fails the WMP. 
	\end{remark}
	
	\begin{remark} Since $\p_Y$ is 1-Lipschitz and $\p_Y(t) \geq \max\{0,t-1\}$, the condition $\p_Y(1)=0$ is equivalent to $\p_Y(t)=\max\{0,t-1\}$ for all $t>0$. These spaces are called metric weak$^*$ Kadec-Klee spaces. In the separable case, they are precisely those spaces which are $(1+\varepsilon)$-isomorphic to a subspace of $c_0$, for every $\varepsilon>0$ \cite{GKL00}.
	\end{remark}
	
	The next result follows the idea of \cite[Theorem 2]{Pellegrino_09}, where the same  is obtained in the case $X^*=\ell_p$, $Y=\ell_q$, $p\neq q$. 
	
	\begin{corollary} \label{cor:strongconv}
		Let $X,Y$ be Banach spaces. Assume that $X$ is separable and 
		$\overline{\delta}_{X^*}^*(t) >  \p_Y(t)$ for all $t>0$. Let $T\colon X^*\to Y$ be a weak$^*$-to-weak continuous operator. Then any non-weak$^*$ null maximizing sequence for $T$ has a convergent subsequence. 
	\end{corollary} 
	
	\begin{proof}
		Let $(x_n)_{n=1}^\infty \subset S_{X^*}$ be a non-weak$^*$ null maximizing sequence for $T$. By extracting a subsequence, we may assume that $(x_n)$ is weak$^*$ convergent to $x\neq 0$ and the limit $\lim_n \norm{x_n-x}$ exists. The proof of Theorem~\ref{TheoremWeakStarMP} shows that $T$ attains its norm at $x$, and moreover 
		\[ \p_Y\left(\frac{\lim_n \norm{x_n-x}}{\norm{x}}\right)=\d^*_{X^*}\left(\frac{\lim_n\norm{x_n-x}}{\norm{x}}\right). \]
		Thus $\lim_n\norm{x_n-x}=0$.
	\end{proof}
	
	Note that, as it is shown in \cite{Pellegrino_09}, there are pairs $(X,Y)$ with the WMP such that the conclusion of Corollary~\ref{cor:strongconv} fails. Indeed consider $X=Y=\ell_2$. Then  $(\frac{e_1+e_n}{\sqrt{2}})_{n\geq1}$ is a non-weakly null maximizing sequence for the identity operator $I\colon \ell_2\to\ell_2$ which has no norm convergent subsequence.

	We will now turn to the case when the modulus of asymptotic uniform smoothness of $X$ is bounded above by the modulus of asymptotic uniform convexity of $Y$. To this aim, we will need the following two lemmata.
	
	\begin{lemma}[Proposition 2.3 in \cite{JLPS02}] \label{CompactOperator}  
		
		Let $X,Y$ be Banach spaces. If there exists $t>0$ such that $\p_X(t) < \d_Y(t)$
		then every linear operator from $X$ to $Y$ is compact. 
	\end{lemma}

	The next one is classical and easy to prove.

	\begin{lemma} \label{CompactAttainment}
		Let $X, Y$ be Banach spaces and let $\tau_Y$ be any Hausdorff topology on $Y$ which is coarser than the norm topology. If $T\colon X^*\to Y$ is a compact operator which is weak$^*$-to-$\tau_Y$ continuous, then $T$ attains its norm.
	\end{lemma}
	\begin{proof}
		Let $(x_\alpha^*)$ be a net in $S_{X^*}$ such that $\lim_\alpha\norm{Tx^*_\alpha}=\norm{T}$, we may assume that $(x^*_\alpha)$ is weak$^*$-convergent to $x^*\in B_{X^*}$. Then $(Tx^*_\alpha)$ is $\tau_Y$-convergent to $Tx^*$. Since the norm topology and $\tau_Y$ agree on the norm-compact set $\overline{T(B_{X^*})}$, we have that $\norm{T}=\lim_\alpha\norm{Tx^*_\alpha}=\norm{Tx^*}$. 
	\end{proof}

	We are now ready to prove the desired theorem.
	
	\begin{theorem} \label{TheoremCompactWMP} 
		Let $X,Y$ be Banach spaces and $\eqnorma$ be an equivalent norm on $X^*$. If there exists $t>0$ such that $\p_{\eqnorma}(t) < \d_Y(t)$, then
		\begin{enumerate}
			\item the pair $(X^*,Y)$ has the weak$^*$-to-weakMP.
			\item If moreover $Y \equiv Z^*$ is a dual space, then $(X^*,Z^*)$ has the weak$^*$-to-weak$^*$MP.
		\end{enumerate}
		In particular, if $X$ is reflexive then $(X^*,Y)$ has the WMP. 
	\end{theorem}
	\begin{proof}
		We only prove the first assertion, the proof of the second one is similar and the last statement is quite obvious. Thanks to Lemma~\ref{CompactOperator}, we know that every bounded linear operator from $(X^*,\eqnorma)$ to $Y$ is compact. 
		Since $\eqnorma$ is an equivalent norm on $X^*$, this readily implies that every bounded linear operator from $X^*$ to $Y$ is compact. 
		According to Lemma~\ref{CompactAttainment}, we deduce that every weak$^*$-to-weak continuous operator from $X^*$ to $Y$ attains its norm. 
	\end{proof}
	
	We highlight some consequences of the previous theorems for particular spaces, more examples will be given in Section~\ref{section-applications}. In all the cases the proof follows from  Theorem~\ref{TheoremWeakStarMP} or Theorem~\ref{TheoremCompactWMP} and Example~\ref{ex:ausauc}. 
	
	\begin{corollary} \label{CorWMP}
		Any of the following conditions ensures that the pair $(X,Y)$ has the WMP: 
		\begin{enumerate}[a)] 
			\item $X$ is a reflexive AUC space with $\d_X(t)>t-1$ (e.g. $X=\ell_p, \ell_p(\Gamma)$) and $Y=c_0$.
			\item $X=\ell_p(\Gamma)$  and $Y=\ell_q(\Gamma')$, $1<p<\infty, 1\leq q<\infty$. 
			\item $X=(\sum_{n=1}^\infty E_n)_p$ where $\dim(E_n)<\infty$ and $Y=(\sum_{n=1}^\infty F_n)_p$ where $\dim(F_n)<\infty$, $1<p<\infty, 1\leq q<\infty$.
			\item $X=\ell_p$ and either $Y=d(w,q)$ (as in Example~\ref{ex:ausauc}~$\mathrm{(v)}$) or $Y= Z^q$ (as in Example~\ref{ex:ausauc}~$\mathrm{(vi)}$), with $1<p\leq q<\infty$.
		\end{enumerate}
	\end{corollary}

	\section{Renorming results} \label{section-renorming}

	First, we wish to underline that for a special class of Banach spaces, namely those $X$ with $\Sz(X) = \omega$, one can find an equivalent norm $\eqnorma$ on $X^*$ and a reflexive Banach space $Y$ such that the pair $\big((X^*,\eqnorma),Y\big)$ has the weak$^*$-to-weak$^*$MP. By $\Sz(X)$ we mean the Szlenk index of $X$ and $\omega$ stands for the first countable ordinal. We refer the reader to \cite{LancienSurvey} for the definition of the Szlenk index as well as its basic properties and main applications to the geometry of Banach spaces. 
	\begin{corollary}
		If $X$ is separable with $\Sz(X) = \omega$, then there exists an equivalent norm $\eqnorma$ on $X^*$ and $q \geq 1$ such that the pair $\big((X^*,\eqnorma),\ell_q\big)$ has the weak$^*$-to-weak$^*$MP. 
		
		In particular if $X\equiv Z^*$ is a separable reflexive space with $\Sz(Z)= \omega$, then there exists an equivalent norm $\eqnorma$ on $X$ and there exists $q \geq 1$ such that the pair $\big((X,\eqnorma),\ell_q\big)$ has the WMP.
	\end{corollary}
	\begin{proof} 
		Let us start by recalling the fundamental renorming result for spaces with Szlenk index equal to $\omega$.  The result is due to H.~Knaust, E.~Odell and Th.~Schlumprecht \cite[Corollary 5.3]{KOS99}: if $X$ is a separable Banach space such that $\Sz(X) =\omega$,  then there exists $p\in (1,\infty)$ such that $X^*$ admits an equivalent $p$-AUC* norm $\eqnorma$, moreover $\d^*_{\eqnorma}(t) \geq (1+t^p)^{\frac{1}{p}}-1$. 
		Let us consider $q>p$. As we already mentioned in Example~\ref{ex:ausauc}, $\p_{\ell_q}(t) = (1+t^q)^{\frac{1}{q}}-1$. Thus for every $t>0$ we have that $\d^*_{\eqnorma}(t) > \p_{\ell_q}(t)$. The result now readily follows from Theorem~\ref{TheoremWeakStarMP}.
	\end{proof}

	Unlike Theorem~\ref{TheoremCompactWMP}, the condition on the moduli in Theorem~\ref{TheoremWeakStarMP} is stated for every $t>0$. This can be problematic in some concrete situations, even if $X$ is $p$-AUC and $Y$ is $q$-AUS with $p<q$.  Nevertheless, it is sometimes possible to define equivalent norms on $X$ and $Y$ for which the required inequality is satisfied for every $t>0$. Our next goal is to outline a setting where such renormings exist.
	
	\smallskip
	
	Recall that a sequence $\{E_n\}$ of finite-dimensional subspaces of a Banach space $X$ is called a \emph{finite-dimensional decomposition} (FDD in short) if every $x\in X$ has a unique representation of the form $x=\sum_{n=1}^\infty x_n$ with $x_n\in E_n$. An FDD determines a sequence $\{P_n\}$ of commuting projections, given by $P_n(\sum_{i=1}^\infty x_i)=\sum_{i=1}^n x_i$, in particular $E_n=(P_n-P_{n-1})X$. The sequence $\{(P_n^*-P_{n-1}^*)X^*\}$ is an FDD of its closed linear span in $X^*$; it is often denoted $\{E_n^*\}$ and called the dual FDD of $\{E_n\}$. In the case when the latter closed linear span is the whole $X^*$ (equivalently, if $\lim_n\norm{P_n^*x^*-x^*}=0$ for every $x^*\in X^*$), it is said that $\{E_n\}$ is a \emph{shrinking FDD}. Next an FDD $\{E_n\}$ is called \textit{boundedly complete} if for every $(x_n)_n \in \{E_n\}$ with $\sup_{n \in \N} \|  \sum_{i=1}^n  x_i\| < \infty $, the series $\sum_{i=1}^{\infty} x_i$ converges in $X$. Any space with a boundedly complete FDD is naturally a dual space and moreover the concepts of shrinking and boundedly complete are dual properties, in the sense that an FDD is shrinking if and only if its dual FDD is boundedly complete. We refer the reader to \cite{Biorthogonal} for more details about FDDs (we also recommend \cite[Section~3.2]{albiackalton} where Schauder bases are considered but the proofs easily extend to FDDs).

	\begin{theorem}\label{thm:renorming}  Let $X, Y$ be Banach spaces with shrinking FDDs. Assume that the norm of $X^*$ is $p$-AUC* and the norm of $Y$ is $q$-AUS for some $1<p\leq q<\infty$. Then there exist equivalent norms $\eqnorma_X$ on $X$ and $\eqnorma_Y$ on $Y$ such that the pair $((X^*,\eqnorma_{X}^*), (Y,\eqnorma_Y))$ has the weak*-to-weakMP.
	\end{theorem}
	
	The proof of Theorem \ref{thm:renorming} follows from the relations between asymptotic uniform convexity and lower estimates, and asymptotic uniform smoothness and upper estimates (see \cite[Proposition 2.10]{JLPS02}). Given $p>1$, an FDD $\{E_n\}$ is said to satisfy an upper (lower, respectively) $p$-estimate provided there is a constant $C > 0$ so that for any sequence of consecutively supported vectors $x_1<\cdots<x_m$ in $X$,
	$\|\sum_{k=1}^m x_k\|\leq C(\sum_{k=1}^m \|x_k\|^p)^{1/p}$ ($\|\sum_{k=1}^m x_k\|\geq C^{-1}(\sum_{k=1}^m \|x_k\|^p)^{1/p}$,  respectively).   
	Recall that a \emph{blocking} of an FDD $\{E_n\}$ is a sequence of the form $\{E_{n_i}+\ldots+E_{n_{i+1}-1}\}$ with $1=n_1<n_2<\ldots$. Note that each blocking of $\{E_n\}$ corresponds to a subsequence of $\{P_n\}$. Thus, each blocking of a shrinking FDD is also shrinking. Moreover, each blocking of $\{(P_n^*-P_{n-1}^*)X^*\}$ is the dual of a blocking of $\{E_n\}$.

	\begin{lemma}\label{lemma:deltaestimate} Let $X$ be a Banach space with a shrinking FDD $\{E_n\}$.
		Assume that $\overline{\delta}^*_{X^*}$ has power type $p$. Then there is a blocking $\{G_n\}$ of $\{E_n\}$ such that the dual FDD $\{G_n^*\}$ satisfies a lower $p$-estimate. 	
	\end{lemma}
	
	\begin{proof}
		Take $C>0$ and $t_0>0$ such that $\overline{\delta}_{X^*}^*(t)\geq Ct^p$ for all $0<t\leq t_0$. First, note that we may assume $t_0\geq 4$. Indeed, if $t_0\leq t\leq 4$ then $\overline{\delta}^*_{X^*}(t)\geq \overline{\delta}^*_{X^*}(t_0)\geq C t_0^p \geq C' t^p$ where $C'=\frac{Ct_0^p}{4^p}$.	
		
		Now, let $0<M<\min\{C, 4^{-p}(3^p-1)\}$. We claim that for each $n\in \mathbb N$ there is $m>n$ such that 
		\[ \norm{x^*+y^*}^p \geq \norm{x^*}^p+M\norm{y^*}^p\]
		whenever $x^*\in \operatorname{span}\{E_i^*: i=1,\ldots, n\}$ and $y^*\in \overline{\operatorname{span}}\{E_i^*: i\geq m\}$. If $\norm{y^*}> 4\norm{x^*}$, then 
		\[ 
		\begin{split}
		\norm{x^*+y^*}^p &\geq \left|\norm{y^*}-\norm{x^*}\right|^p = \norm{y^*}^p \left|1-\frac{\norm{x^*}}{\norm{y^*}}\right|^p\geq \frac{3^p}{4^p}\norm{y^*}^p \\
		&\geq \norm{x^*}^p+ \frac{3^p-1}{4^p}\norm{y^*}^p 
		\end{split}
		\]
		so we may assume $\norm{y^*}\leq 4\norm{x^*}$. Moreover, by homogeneity, we may assume $x^*\in S_{X^*}$.
		
		Assume that the claim does not hold. Then there exists $n$ such that for each $m>n $ there is $x_m^*\in \operatorname{span}\{E_i: i=1,\ldots, n\}$ and $y_m^* \in \overline{\operatorname{span}}\{E_i: i\geq m\}$ such that  $\norm{x_m^*}=1$, $\norm{y_m^*}\leq 4$, and $\norm{x_m^*+y_m^*}^p <\norm{x_m^*}^p+M\norm{y_m^*}^p$. By extracting a subsequence we may assume $x_m^*\to x^*\in S_{X^*}$. Since $\{E_n\}$ is an FDD for $X$, the sequence $\{y_m^*\}_{m\geq n}$ is weak$^*$-null. Moreover, $\limsup\norm{y_m^*}\leq 4$. Thus,
		\[ \lim_{m\to\infty}\norm{x^*+y_m^*} \geq 1+\overline{\delta}^*_{X^*}(\limsup_m\norm{y_m^*})\geq 1+C\limsup_m\norm{y_m^*}^p.\]
		On the other hand, 
		\[
		\begin{split}
		\lim_{m\to\infty}\norm{x^*+y_m^*} &= \lim_{m\to\infty}\norm{x_m^*+y_m^*}\leq (1+M\limsup_m\norm{y_m^*}^p)^{1/p}\\
		&\leq 1+M\limsup_m\norm{y_m^*}^p.
		\end{split}
		\]
		From that, we got $C\leq M$, a contradiction. 
		
		It follows from the claim that there is an increasing sequence $1=n_0<n_1<n_2<...$ such if $x^*\in \operatorname{span}\{E_i^*: 1\leq i\leq n_i\}$ and $y^*\in \overline{\operatorname{span}}\{E_i^* :  i\geq n_{i+1}\}$ then $\norm{x^*+y^*}^p\geq \norm{x^*}^p+M\norm{y^*}^p$. It follows easily that the blocking $\{F_n^*\}$ of $\{E_n^*\}$ given by $F_i^* = \operatorname{span}\bigcup_{k=n_i}^{n_{i+1}-1} E_k^*$ satisfies a skipped lower $p$-estimate. Moreover, $\{F_n^*\}$ is boundedly complete since it is a blocking of a boundedly complete FDD. Finally, a result due to Johnson \cite{Johnson77} (see also \cite[Lemma 2.51]{Biorthogonal}) provides a further blocking $\{G_n^*\}$ of $\{F_n^*\}$ satisfying a lower $p$-estimate.
	\end{proof}
	
	It is well known that an FDD satisfies a lower estimate if and only if the dual FDD satisfies an upper estimate \cite[Fact~2.19]{Biorthogonal}. In fact, sometimes we can keep the constants, as the following lemma shows. 
	
	\begin{lemma}\label{lemma:dualestimate} Let $\{E_n\}$ be a shrinking FDD. Assume that the dual FDD $\{E_n^*\}$ satisfies a lower $p$-estimate with constant $1$. Then $\{E_n\}$ satisfies an upper $p'$-estimate with constant $1$, where $1/p+1/p'=1$. 
	\end{lemma}
	
	\begin{proof}
		Let $x_1,\ldots , x_n\in X$ be consecutively supported vectors such that $\| \sum_{i=1}^n x_i \| = 1$. Using Hahn--Banach theorem, we may consider $x^*\in S_{X^*}$ such that $\langle x^*, \sum_{i=1}^n x_i\rangle= 1$. Since $\{E_n\}$ is shrinking, $\{E_n^*\}$ is an FDD in $X^*$ so we may write $x^* = \sum_{i=1}^{\infty} x_i^*$
		with $\supp(x_i^*)=\supp(x_i)$ whenever $i \leq n$. For $i >n$, we let $x_i = 0$ in what follows:
		\begin{align*} 
		1 &= \langle x^*, \sum_{i=1}^n x_i\rangle = \sum_{i=1}^n x_i^*(x_i) =  \sum_{i=1}^{\infty} x_i^*(x_i) \leq \sum_{i=1}^\infty \norm{x_i^*}\norm{x_i} \\
		&\leq   \left(\sum_{i=1}^{\infty}\norm{x_i^*}^p\right)^{\frac{1}{p}}\left(\sum_{i=1}^{\infty} \norm{x_i}^{p'}\right)^{\frac{1}{p'}}
		\leq \Big\|\sum_{i=1}^{\infty} x_i^*  \Big\| \left(\sum_{i=1}^n \norm{x_i}^{p'}\right)^{\frac{1}{p'}} \\
		&\leq \left(\sum_{i=1}^n  \norm{x_i}^{p'}\right)^{\frac{1}{p'}}. 
		\end{align*}
	\end{proof}

	\begin{proof}[Proof of Theorem~\ref{thm:renorming}]
		By Lemma~\ref{lemma:deltaestimate} there is a blocking $(E_n)_n$ of the FDD in $X$ whose dual FDD satisfies a lower $p$-estimate. Now, consider the norm on $X^*$ given by
		\[\norm{x^*}_{(p)}=\sup\left\{\left(\sum_{i=1}^k \norm{(P_{n_{i+1}}^*-P_{n_i}^*)(x^*)}^p\right)^{1/p} : 1\leq n_1<n_2<\ldots<n_k\right\}\]
		where $\{P_n\}$ are the associated projections to $\{E_n\}$. This is the norm considered by Prus in \cite[Remark~3.2 and Lemma~3.3]{Prus87}, where it is proved that it is an equivalent norm on $X^*$ and that the FDD $\{E_n^*\}$ satisfies a lower $p$-estimate with constant $1$. Moreover, the function $\norm{\cdot}_{(p)}$ is weak*-lower semicontinuous since it is the supremum of weak*-lower semicontinuous functions, so $\norm{\cdot}_{(p)}$ is a dual equivalent norm on $X^*$. It follows that $\d^*_{(X^*, \norm{\cdot}_{(p)})}(t)\geq (1+t^p)^{1/p}-1$ for all $t>0$. 
		
		Now, we focus on the space $Y$. Since the norm is $q$-AUS, the dual norm is $q'$-AUC*, where $1/q+1/q'=1$. Lemma~4.3 provides a blocking $\{F_n\}$ of the FDD of $Y$ such that the dual FDD satisfies a lower $q'$-estimate in $Y^*$. The argument above provides a dual equivalent norm $\norm{\cdot}_{(q')}$ on $Y^*$ such that $\{F_n^*\}$ satisfies a lower $q'$-estimate with constant one. Then, $\{F_n\}$ satisfies an upper $q$-estimate in $Y$ endowed with the predual norm $\norm{\cdot}_{(q')_*}$ with the same constant (one) by Lemma~\ref{lemma:dualestimate}. Thus,  $\p_{(Y,\norm{\cdot}_{(q')_*})}(t)\leq (1+t^q)^{1/q}-1$ for all $t>0$. It follows that, given $t\geq 0$, either $\d^*_{(X^*, \norm{\cdot}_{(p)})}(t)>\p_{(Y,\norm{\cdot}_{(q')_*})}(t)$, or $p=q$ and $\d^*_{(X^*, \norm{\cdot}_{(p)})}(t)=\p_{(Y,\norm{\cdot}_{(q')_*})}(t)=(1+t^p)^{1/p}-1$, which is greater than $t-1$ for $t\geq 1$. Finally, the conclusion follows from Theorem~\ref{TheoremWeakStarMP}.   
	\end{proof}
	
	Note that we may avoid one of the renormings in Theorem~\ref{thm:renorming} in case that $\d_{X^*}^*(t)=(1+t^p)^{1/p}-1$ or $\p_{Y}(t)=(1+t^q)^{1/q}-1$. As a consequence, we get for instance that for every $1<p<\infty$ there is an equivalent norm $\eqnorma$ on the Tsirelson space $\mathrm{T}$ so that the pair $((\mathrm{T}, \eqnorma), \ell_p)$ has the WMP.

	\section{Application to some classical Banach spaces}
	\label{section-applications}

	\subsection{The case of \texorpdfstring{$\ell_1$}{l1} and Schur spaces}
	
	Recall that a Banach space $Y$ is said to have the Schur property if every weakly convergent sequence is also norm-convergent. Our first goal is to show that Schur spaces are the best range spaces for the WMP. In the sequel, we will need to following well-known fact, its easy proof is left to the reader.
	
	\begin{lemma} \label{lemmaNA}
		Let $T :X \to Y$ be a bounded operator. If $T$ attains it norm then $T^* : Y^* \to X^*$ attains its norm. Conversely, if $T^*$ attains its norm and if $X$ is reflexive, then $T$ also attains its norm. 
	\end{lemma}
	
	\begin{proposition} \label{prop-Schur-range}
		If $X$ is reflexive and $Y$ is Schur, then every operator $T\colon X\to Y$ is compact. Therefore $(X,Y)$ has the WMP and $(Y^*, X^*)$ has the weak$^*$-to-weak$^*$MP.
	\end{proposition}

	\begin{proof}
		The first part of the statement can be obtained as a particular case of each of the following facts:
		\begin{itemize}
			\item Every weakly precompact operator taking values in a Schur space is compact (obvious). Recall that an operator $T\colon X \to Y$ is called weakly precompact if $T(B_X)$ is weakly precompact, i.e., every sequence in $T(B_X)$ admits a weakly Cauchy subsequence.
			\item Every completely continuous operator defined on a Banach space not containing $\ell_1$ is compact (apply Rosenthal’s $\ell_1$-theorem). Recall that an operator $T\colon X \to Y$ is called completely continuous if it is weak-to-norm sequentially continuous.
		\end{itemize}
		The second part of the statement follows from the fact that any compact operator on a reflexive space attains its norm and Lemma~\ref{lemmaNA}.
	\end{proof}

	The last proposition allows us to deduce for instance that the pair  $(\ell_2, \ell_1)$ has the WMP while $(\ell_\infty, \ell_2)$ has the weak$^*$-to-weak$^*$MP.
	Next, it is well known that $\ell_1 \equiv c_0^*$ is weak$^*$ asymptotically uniformly convex with modulus $\d_{\ell_1}^*(t)=t$, and the same is true for $Z^*$ instead of $\ell_1$ whenever $Z$ is a subspace of $c_0$. Consequently,
	the next result is a direct consequence of Theorem~\ref{TheoremWeakStarMP}. Let us mention that if a dual Banach space $X \equiv Z^*$ has modulus $\d_{X}^*(t)=t$,
	then $Z$ is asymptotically uniformly flat (that is there exists $t_0>0$ such that for every $t<t_0$, $\p_Z(t) = 0$). 
	This is actually equivalent to the fact that $Z$ is isomorphic to a subspace of $c_0$ (see \cite[Theorem~2.9]{JLPS02}).
	\begin{corollary} \label{Cor-l1-domain}
		If $Z$ is a closed subspace of $c_0$,
		then the pair $(Z^{*},Y)$ has the W$^*$MP for every Banach space $Y$. 
		In particular $(\ell_1,Y)$ has the W$^*$MP for every Banach space $Y$ \emph{(with the weak$^*$ topology given by $c_0$)}.
	\end{corollary}

	However, the previous result is not true in general if one considers another weak$^*$ topology on $\ell_1$, or if one replaces $\ell_1$ by any dual space with the Schur property. This is shown by the next examples, we thank Gilles Lancien and Antonin Proch\'azka for presenting the first one to us. Notice that this proves in particular that the W$^*$MP is not an isomorphic property, in the sense that there are isomorphic spaces $X_1$, $X_2$ so that only one of the pairs $(X_i^*, Y)$ has the W*MP.

	\begin{example}\label{ex:failsw*w*MP}
		Consider $T\colon \mathcal{C}[0,\omega]^* \to c_0^*$ such that $T(e_n) = (1-\frac{1}{n})e_n$ and $T(e_{\omega}) = 0$. Of course $c_0^* \equiv \ell_1$ while $ \mathcal{C}[0,\omega]^* \equiv \ell_1([0,\omega]) \equiv \ell_1$. It is readily seen that $\|T\| = 1$ and that $T$ does not attain its norm. Now the sequence $(e_n)_n \in \mathcal{C}[0,\omega]^*$ is a maximizing sequence for $T$, but this sequence is not weak$^*$ null since it converges to $e_\omega$. Thus the pair $(\mathcal{C}[0,\omega]^* , c_0^*)$ fails the W$^*$MP. In fact, $T$ is the dual operator of $S\colon c_0 \to \mathcal{C}[0,\omega]$ given by $Sx(n)=(1-\frac{1}{n})x_n$ and $Sx(\omega) = 0$ for every $x=(x_n)_n \in c_0$. Therefore the pair $(\mathcal{C}[0,\omega]^* , c_0^*)$ even fails the weak$^*$-to-weak$^*$MP.
	\end{example}
	
	A very similar operator $T\colon \mathcal{C}[0,\omega]^* \to \R$, defined by $T(e_n) = (1-\frac{1}{n})$ and $T(e_{\omega}) = 0$, actually yields that the pair $(\mathcal{C}[0,\omega]^* , \R)$ fails the W$^*$MP.  We choose to present another example of a different nature.

	\begin{example}  \label{Example1-Schur}
		There exists a dual space $Z=X^*$ with the Schur property such that the pair $(X^*, \R)$ fails the W$^*$MP. In fact $Z$ is going to be the Lipschitz-free space $\F(M)$ where $M = \{0 \}\cup \{x_n  \; : \;  n \in  \N \} \subset c_0$ is the pointed metric space given by $x_1 = 2e_1$, and $x_n = e_1+(1+\frac{1}{n})e_n$ for $n \geq 2$, where $(e_n)$ is the canonical basis of $c_0$. Then it is known \cite{Kalton04} that $\F(M)$ has the Schur property as $M$ is uniformly discrete and bounded. Indeed, the sequence $(\delta(x_n))$ is a Schauder basis for $\F(M)$ which is equivalent to the standard $\ell_1$-basis. Therefore $\F(M)$ is actually isomorphic to $\ell_1$ (see \cite[Proposition 4.4]{Kalton04} for more details). It is also known \cite[Example~5.6]{GPPR_2018} that $\F(M)\equiv X^*$ where
		\[ X = \{f \in \mathrm{Lip}_0(M) \; : \; \lim\limits_{n \to \infty} f(x_n) = \frac{1}{2} f(x_1)\}\]
		is isomorphic to $c_0$. It is readily seen that the sequence $((1+\frac{1}{n})^{-1}\delta(x_n))_n$ weak$^*$ converges to $\frac{1}{2}\delta(x_1) \neq 0$. Now let $f \colon M \to \R$ be the Lipschitz map defined by $f(0)=f(x_1)=0$ and $f(x_n)=1$ for every $n \geq 2$. It is obvious that $\|f\|_L = \lim\limits_{n \to \infty} \frac{f(x_n) -f(0)}{d(x_n,0)} = 1$. Now using the linearization property of free spaces, we may consider the bounded operator $\overline{f} \colon \F(M) \to \R$ such that $\overline{f}(\delta(x_n)) = f(x_n)$. Observe that $((1+\frac{1}{n})^{-1}\delta(x_n))_n$  is a maximizing sequence which is not weak$^*$-null. To finish, we just need to show that $\overline{f}$ does not attain its norm. Assume that $\overline{f}(\mu)=1$ for $\mu=\sum_{n=1}^\infty a_n \delta(x_n)\in S_{\mathcal F(M)}$. First note that 
		\[ 1=\overline{f}(\mu) = \sum_{n=2}^\infty a_n.\]
		Let $A=\{n\geq 2:  a_n>0\}$ and consider the one-Lipschitz function $g\colon M\to\mathbb R$ given by $g(x_n)=1$ if $n\in A$ and $g(x_n)=0$ otherwise. By evaluating the linear extension of $g$ to $\mathcal F(M)$ at $\mu$ we get $\sum_{n\in A} a_n\leq 1$. It follows that $a_n\geq 0$ for all $n\geq 2$. 
		\medskip
		
		Finally, consider the one-Lipschitz function $j\colon M\to\mathbb R$ given by $j(0)=j(x_1)=0$ and $j(x_n)=1+1/n$. Then, 
		\begin{align*}
		1&=\norm{\mu} \geq \<\overline{j}, \mu\> 
		= \sum_{n=2}^\infty a_n\left(1+\frac{1}{n}\right) = 1+ \sum_{n=2}^\infty \frac{a_n}{n}.
		\end{align*}
		Thus, $a_n=0$ for all $n\geq 2$. Then $\overline{f}(\mu)=0$, a contradiction. 
	\end{example}
	
	Notice that the Banach space $\mathcal F(M)$ in the previous example is not AUC*: indeed $\frac{1}{2}\delta(x_1)$ is an extreme point of $B_{\mathcal F(M)}$ \cite[Theorem~3.2]{AP20} but not a preserved extreme point \cite[Theorem~4.1]{AG19}, in particular it does not belong to slices of small diameter.

	\subsection{Dunford--Pettis spaces}

	We recall that a Banach space $Y$ has the Dunford--Pettis property if and only if for every sequence $(y_n)_{n}$ in $Y$ converging weakly to $0$ and every sequence $(y_n^*)_n$ in $Y^*$ converging weakly to 0, the sequence of scalars $(y_n^*(y_n))_n$ converges to 0. 
	It is well-known and easy to see that the weak convergence of $(y_n)_n$ (or of $(y_n^*)_n$) can be changed by weakly Cauchy convergence in this definition (see \cite{DiestelDP} e.g.).
	Classical spaces with the Dunford--Pettis property include Schur spaces, $c_0$, $\mathcal{C}(K)$ spaces (where $K$ is a compact Hausdorff space) and $L_1(\mu)$ (where $\mu$ is a $\sigma$-finite measure), see e.g.~\cite{DiestelDP}. However reflexive (infinite dimensional) spaces never have it.
	
	One goal of the section is to take advantage of the Dunford--Pettis property to analyse the possible relationships between $(X,Y)$ having the WMP and $(Y^*,X^*)$ having the weak$^*$-to-weak$^*$MP. Let us begin with the next observation.
	\begin{lemma} \label{LemmaMaxSeqDUAL}
		If $X$ is reflexive, $Y$ has the Dunford--Pettis property and $T\colon X \to Y$ is a nonzero operator, then every  maximizing sequence for $T^*\colon Y^* \to X^*$ is non-weakly null.\\
		Moreover, if $T^*$ admits a weakly Cauchy maximizing sequence then both $T$ and $T^*$ attain their norm.
	\end{lemma}
	
	\begin{proof}
		Let us prove the first statement. Let $(y_n^*)_n \subset S_{Y^*}$ be a maximizing sequence for $T^*$. We define $(x_n)_n \subset S_{X}$ so that $T^*y_n^*(x_n) > \|T^*y_n^*\| - \frac{1}{n}$. It is easy to see that 
		\[\|Tx_n\| \geq |y_n^*(Tx_n)| = T^*y_n^*(x_n) > \|T^*y_n^*\| - \frac{1}{n},\] 
		and taking the limit $n \to \infty$ yields $\lim\limits_{n \to \infty} \|Tx_n\| \geq \|T^*\|$. Since $\|T^*\| = \|T\|$ we deduce that $(x_n)_n$ is a maximizing sequence for $T$.
		
		By weak-compactness, we may extract a subsequence of $(x_n)_n$ (still denoted the same way) which weakly converges to some $x$. Now as $T$ is weakly continuous, we have that $(T x_n)_n$ weakly converges to $Tx$. We claim that $(y_n^*)_n$ is not weakly null. Indeed, aiming for a contradiction, assume that $(y_n^*)_n$ is weakly null. As $Y$ has the Dunford--Pettis property, we have that $\lim\limits_{n \to \infty} y_n^*(Tx_n) =\lim\limits_{n \to \infty} y_n^*(Tx_n-Tx) = 0$, contraducting  our definition of $(y_n^*)_n$ and $(x_n)_n$.  
		\smallskip
		
		Finally, if $(y_n^*)_n$ is a weakly Cauchy maximizing sequence for $T^*$, we may build an associated maximizing sequence $(x_n)_n$ for $T$ as described above. Up to extracting a subsequence, we may assume that $(x_n)_n$ is weakly convergent to some $x \in X$. Of course, $\|x\| \leq 1$ by lower semi-continuity of the norm. Next, since $Y$ has the Dunford--Pettis property and $T$ is weakly continuous, $\lim\limits_n \langle y_n^* , Tx_n -Tx \rangle  = 0$. By construction we know that  $\lim\limits_n \langle y_n^* , Tx_n \rangle  = \|T\|$, which in turn implies that $\lim\limits_n \langle y_n^* , Tx \rangle  = \|T\|$, and so we obtain that $\|Tx\| \geq \|T\|$. Consequently $\|x\|=1$ and $T$ attains its norm at $x$. The fact that $T^*$ also attains its norm follows from Lemma~\ref{lemmaNA}.
	\end{proof}
	
	\begin{remark}
		By Rosenthal's $\ell_1$ theorem, 
		if $T^*$ does not admit a weakly Cauchy maximizing sequence (as it is assumed in the second part of Lemma~\ref{LemmaMaxSeqDUAL}), then every maximizing sequence for $T^*$ has a subsequence equivalent to the $\ell_1$ canonical basis. Let us point out that $Y^*$ contains a copy of $\ell_1$ whenever $Y$ is infinite-dimensional and has the Dunford--Pettis property (see \cite[Corollary~7]{Castillo} e.g.). 
	\end{remark}

	\begin{remark}
		Assume this time that $X$ has the Dunford--Pettis property while $Y$ is reflexive. Similarly as above, one can show that every bounded operator $T\colon X \to Y$ admits a non weakly-null maximizing sequence. However, since $X$ is non reflexive, there is no chance that the pair $(X,Y)$ enjoys the WMP.
	\end{remark}

	We recall that a Banach space $Y$ has the weak$^*$-Dunford--Pettis property if for all weakly null sequences $(y_n)_n$ in $Y$ and all weak$^*$ null sequences $(y_n^*)_n$ in $Y^*$, the sequence $(y_n^*(y_n))_n$  converges to 0. Let us provide some results from \cite{Jaramillo2000} concerning this property. It is clear that the weak$^*$-Dunford--Pettis property implies the classical Dunford--Pettis property. For instance $\ell_1$ (and any Schur space), $\ell_\infty$ and $\ell_1 \oplus \ell_\infty$ have this property. Moreover, if $Y$ has the weak$^*$-Dunford--Pettis property and if the unit ball of $Y^*$ is weak$^*$-sequentially compact (for instance when $Y$ is separable), then $Y$ is a Schur space.
	A typical example of a space having the Dunford--Pettis property but lacking the weak$^*$-Dunford--Pettis property is $c_0$. 	The following result is a modified version of Lemma~\ref{LemmaMaxSeqDUAL}. 
	\begin{lemma} \label{LemmaMaxSeqDUALstar}
		If $X$ is reflexive, $Y$ has the weak$^*$-Dunford--Pettis property and $T\colon X \to Y$ is a nonzero operator, then every maximizing sequence for $T^* \colon Y^* \to X^*$ is non-weak$^*$ null.\\
		If moreover the unit ball of $Y^*$ is weak$^*$-sequentially compact (in particular Y is a Schur space) then every maximizing sequence for a nonzero operator $T\colon X \to Y$ is non-weakly null.
	\end{lemma}
	\begin{proof}
		The proof of the first statement follows the same lines as in Lemma~\ref{LemmaMaxSeqDUAL}, so we shall omit the details. Let us prove the second statement. Suppose that $Y^*$ is weak$^*$-sequentially compact and let $(x_n)_n$ be a maximizing sequence for a non-zero operator $T:X \to Y$.  Since $X$ is reflexive, up to an extraction, we may assume that $(x_n)_n$ weakly converges to some $x \in X$. Let $(y_n^*)_n$ be a sequence in $S_{Y^*}$ such that 
		\[ T^*y_n^*(x_n) = y_n^*(Tx_n) > \|Tx_n\| - \frac{1}{n},\]
		then $(y_n^*)_n$ is a maximizing sequence for $T^*$.
		By assumption, up to an extraction, we may assume that $(y_n^*)_n$ weak$^*$ converges to some $y^* \in Y$. On the one hand $\lim\limits_{n} \langle y_n^*, Tx_n \rangle = \|T\|$ by construction. On the other hand, $\lim\limits_{n} \langle y_n^*, Tx_n \rangle = \langle y^* , Tx \rangle$ since $Y$ has the weak$^*$-Dunford--Pettis property. This clearly yields that $(x_n)_n$ is a non-weakly null maximizing sequence for $T$.
	\end{proof}

	\begin{proposition} \label{prop-DP-Duality}
		Let $X$ be a reflexive space and let $Y$ be a Banach space with the weak$^*$-Dunford--Pettis property.
		Then the following assertions are equivalent:
		\begin{enumerate}
			\item[(i)] $(Y^*,X^*)$ has the weak$^*$-to-weak$^*$MP,
			\item[(ii)] Every dual operator $T^* \colon Y^* \to X^*$ attains its norm,
			\item[(iii)] Every operator $T \colon X \to Y$ attains its norm.
		\end{enumerate}
		In particular, any of the above imply that $(X,Y)$ has the WMP.
	\end{proposition}

	\begin{proof}
		According to Lemma~\ref{LemmaMaxSeqDUALstar}, for every bounded operator $T \colon X \to Y$, there exists a non-weak$^*$ null maximizing sequence for $T^* \colon Y^* \to X^*$. So if we assume $(i)$, by definition we clearly obtain $(ii)$. Also, it is trivial that $(ii) \implies (i)$. Now $(ii) \iff (iii)$ since $X$ is reflexive, see Lemma~\ref{lemmaNA}. Finally, the assertion $(iii)$ readily yields that $(X,Y)$ has the WMP.
	\end{proof}
	
	\begin{remark}
		If moreover the unit ball of $Y^*$ is assumed to be weak$^*$-sequentially compact then any of the previous statements is equivalent to: 
		\begin{enumerate}
			\item[(iv)] $(X,Y)$ has the WMP.
		\end{enumerate}
		Indeed, in this case $Y$ is a Schur space so that Proposition~\ref{prop-Schur-range} ensures that $(X,Y)$ has the WMP while $(Y^*,X^*)$ has the weak$^*$-to-weak$^*$MP. Of course this could also be proved directly $(iv) \implies (i)$: If $(X,Y)$ has the WMP, we consider $(y_n^*)_n$ a non-weak$^*$ null maximizing sequence for $T^* \colon Y^* \to X^*$. By assumption, $Y^*$ is weak$^*$-sequentially compact so we may apply the second part of Lemma~$\ref{LemmaMaxSeqDUALstar}$ which provides a non-weakly null maximizing sequence for $T$. Therefore $T$ attains its norm, and so does $T^*$ thanks to Lemma~\ref{lemmaNA}.
	\end{remark}
	
	Note that the hypothesis on the weak$^*$-Dunford--Pettis property is actually needed in Propostion~\ref{prop-DP-Duality}. Indeed, the pair $(\ell_1, \mathbb R\oplus_1 \ell_2)$ has the W*MP (and so the weak*-to-weak*MP) by Theorem~\ref{TheoremWeakStarMP}. However, $(\mathbb R\oplus_\infty \ell_2, c_0)$ fails the WMP as we have seen in Remark~\ref{rem:failing}.
	
	On the other hand, there are spaces $X, Y$ so that $(Y^*, X^*)$ has the WMP (so the weak$^*$-to-weak$^*$MP) but $(X, Y)$ and fails the WMP: just take $Y=\mathbb R$ and $X$ any non-reflexive space. 
	The reverse situation is also of particular interest:

	\begin{question} 
		If $(X,Y)$ has the WMP, does it follow that $(Y^*,X^*)$ has the  weak$^*$-to-weak$^*$MP? 
	\end{question}

	It is known that a dual Banach space $X^*$ has the Schur property if and only if $X$ has the Dunford--Pettis property and does not contain any isomorphic copy of $\ell_1$. So, if $Y$ is reflexive, $X$ has the Dunford--Pettis property and does not contain any isomorphic copy of $\ell_1$, then $(Y^*,X^*)$ has the WMP. Notice that it is not possible to remove the assumption on the non-containment of $\ell_1$ for $X$. Indeed, $\ell_1$ has the Dunford--Pettis property and $\ell_2$ is of course reflexive but the pair $(\ell_2 , \ell_\infty)$ fails the WMP, as we shall explain in the next example. In fact our example tells us a bit more. 
	Indeed, we prove that $(\ell_2,\ell_\infty)$ even fails the weak$^*$-to-weak$^*$MP, while $(\ell_2,c_0)$ has the WMP and $(\ell_1,\ell_2)$ has the W$^*$MP (so the weak$^*$-to-weak$^*$MP).
	It was suggested to us by R. Aron to consider the following variation of the WMP.
	\begin{definition}
		We say that a pair $(X,Y)$ has the bidual WMP if for every operator $T \colon X \to Y$, the existence of a non weakly-null maximizing sequence for $T$ implies that $T^{**} \colon X^{**} \to Y^{**}$ attains its norm.
	\end{definition}
	It is obvious from the definitions that if $(X^{**},Y^{**})$ has the weak$^*$-to-weak$^*$MP then $(X,Y)$ has the bidual WMP. However, the  converse does not hold in general as witnessed by the following example.
	\begin{example} 
		First of all, $(\ell_2, c_0)$ has the WMP thanks to Corollary~\ref{CorWMP}~a), thus it also has the bidual WMP. Next $(\ell_1,\ell_2)$ has the W$^*$MP thanks to Corollary~\ref{Cor-l1-domain}.
		It only remains to prove that $(\ell_2,\ell_\infty)$ fails the weak$^*$-to-weak$^*$MP (and consequently also fails the WMP and the W$^*$MP).
		
		Let $T \colon \ell_2 \to \ell_{\infty}$ defined for every $x=(x_n)_n \in \ell_2$ by 
		\[Tx = \left(x_1 , \; x_1 + \Big(1-\frac{1}{2}\Big)x_2 , \;  x_1 +   \Big(1-\frac{1}{3}\Big)x_3  , \; \ldots\; , \; x_1 + \Big(1-\frac{1}{n}\Big)x_n  , \; \ldots\right).\] 
		In other words, $Te_1 = \car$ (the constant equal to 1 sequence) and $Te_n = (1-\frac{1}{n})e_n$. Our claim follows from the items below: 
		\begin{itemize}
			\item $\|T\| = \sqrt{2}$. Indeed, 
			\begin{align*} \norm{Tx}_\infty&=\sup\big\{|x_1+(1-1/n)x_n| : n\geq 1 \big\}\\
			&\leq 
			\sup\big\{(|x_1|^2+|x_n|^2)^{1/2}(1+(1-1/n)^2)^{1/2}: n\geq 2\big\} \\
			&\leq \sqrt{2}\norm{x}_2.
			\end{align*}
			Moreover, 
			$$\Big\|T\Big(\frac{1}{\sqrt{2}}(e_1+e_n)\Big) \Big\|\to \sqrt{2},$$ so $\norm{T}=\sqrt{2}$ the sequence $\big(\frac{1}{\sqrt{2}}(e_1+e_n)\big)_n$ is a normalized maximizing sequence which is not weakly$^*$ null (it weakly$^*$ converges to $\frac{1}{\sqrt{2}}e_1$).
			\item $T$ is weak*-to-weak*-continuous: one can easily check that $T=S^*$ where $S\colon \ell_1\to\ell_2$ is given by $Sx=(\sum_{n=1}^\infty x_n)e_1+\sum_{n=1}^\infty (1-1/n)x_n e_n$.  
			\item $T$ does not attain its norm: if $\norm{Tx}_\infty=\sqrt{2}$ for some $x$ with $\norm{x}_2=1$, then the above estimation implies that $\limsup_n (|x_1|^2+|x_n|^2)^{1/2} = 1$. Thus $|x_1|=1$, which implies that $x=\pm e_1$, a contradiction.
		\end{itemize}
	\end{example}

	\subsection{The case of the James sequence spaces}

	Let $p\in (1,\infty)$. We now recall the definition and some basic properties of the James space $\J_p$. We refer the reader to \cite[Section 3.4]{albiackalton} and references therein for more details on the classical case $p=2$. The James space $\J_p$ is the real Banach space of all sequences $x=(x(n))_{n\in \N}$ of real numbers with finite $p$-variation and satisfying $\lim_{n \to \infty} x(n) =0$. The space $\J_p$ is endowed with the following norm
	\[\|x\|_{\J_p} = \sup  \Big \{  \big (\sum_{i=1}^{k-1} |x(p_{i+1})-x(p_i)|^p \big )^{1/p}     \; \colon \; 1 \leq p_1 < p_2 < \ldots < p_{k} \Big \}. \]
	This is the historical example, constructed for $p=2$ by R.~C.~James, of a quasi-reflexive Banach space which is isomorphic to its bidual. In fact $\J_p^{**}$ can be seen as the space of all sequences $x=(x(n))_{n\in \N}$ of real numbers with finite $p$-variation, which is $\J_p \oplus \R \car$, where $\car$ denotes the constant sequence equal to $1$.
	The standard unit vector basis $(e_n)_{n=1}^{\infty}$ ($e_n(i)=1$ if $i = n$ and $e_n(i)=0$ otherwise) is a
	monotone shrinking basis for $\J_p$. Hence, the sequence $(e_n^*)_{n=1}^{\infty}$ of the associated coordinate
	functionals is a basis of its dual $\J_p^*$. Then the weak topology $\sigma(\J_p,\J_p^*)$ is easy to
	describe: a sequence $(x_n)_{n=1}^{\infty}$ in $\J_p$ converges to $0$ in the $\sigma(\J_p,\J_p^*)$ topology
	if and only if it is bounded and $\lim_{n \to \infty} x_n(i)=0$ for every $i \in \N$.
	For $x\in \J_p$, we define $\supp x = \{i \in \N \, : \, x(i) \neq 0 \}$. 
	
	The detailed proof of the following proposition can be found in \cite[Corollary~2.4]{Netillard}. 
	This a consequence of the fact that the basis of $\mathcal J_p$ satisfies an upper $p$-estimate.

	\begin{proposition}\label{Jsmooth} There exists an equivalent norm $|\cdot|$ on $\J_p$ such that it has the following property: for any $x,y \in \J_p$ such that $\max \supp x < \min \supp y$, we have that
		\[|x+y|^p\leq |x|^p+|y|^p.\]
		In particular, the modulus of asymptotic uniform smoothness of $\widetilde{\J_p}:=(\J_p,|\cdot|)$ is $\p_{\widetilde{\J_p}}(t) \leq (1+t^p)^{\frac{1}{p}} - 1$ for all $t\geq 0$.
	\end{proposition}

	There is also a natural weak$^*$ topology on $\J_p$. Indeed, the summing basis $(s_n)_{n=1}^{\infty}$ ($s_n(i)=1$ if $i \leq n$ and $s_n(i)=0$ otherwise) is a monotone and boundedly complete basis for $\J_p$. Thus, $\J_p$ is naturally isometric to a dual Banach space: $\J_p = X^*$ with $X$ being the closed linear span of the biorthogonal functionals $(e_n^* - e_{n+1}^*)_{n=1}^{\infty}$ in $\J_p^*$ associated with $(s_n)_{n=1}^{\infty}$. Note that $X=\{x^*\in \J_p^*,\ \sum_{n=1}^\infty x^*(n)=0\}$. Thus, a sequence $(x_n)_{n=1}^{\infty}$ in $\J_p$ converges to $0$ in the $\sigma(\J_p,X)$ topology if and only if it is bounded and $\lim_{n \to \infty} \big(x_n(i) - x_n(j)\big) = 0$ for every $i \neq j \in \N$. 
	
	The next lemma is classical, we include its proof for completeness.
	
	\begin{lemma}\label{Jconvex} 
		Let $(x_n)_{n=1}^\infty$ be a weak$^*$-null sequence in $\J_p$. Then, for every $x \in \mathcal J_p$, we have 
		\[ \limsup_{n\to\infty} \norm{x +x _n}^p \geq \|x\|^p + \limsup_{n\to\infty} \|x_n\|^p. \]
		Consequently, the modulus of weak$^*$ asymptotic uniform convexity of $\J_p$ is given by 
		\[\d^*_{\J_p}(t) \geq (1+t^p)^{\frac{1}{p}} - 1.\]
	\end{lemma}
	\begin{proof} 
		Let $x=(x(n))_n \in \J_p$, $(x_n)_{n=1}^\infty$ be a weak$^*$-null sequence and $\ep >0$ be fixed. We may and do assume that $\|x\|=1$ and $x$ is finitely supported. Let $(p_i)_{i=1}^{N} \subset \N$ be an increasing family such that \[\sum_{i=1}^{N-1} |x(p_{i+1})-x(p_i)|^p =1\]
		and take $N_1 > p_N$ such that $x(n)=0$ for $n\geq N_1$. 
		Now, since $(x_n)_{n=1}^\infty$ is a weak$^*$-null sequence, it is bounded by some $M>0$ and there exists $N_2 \in \N$ large enough so that for every $n \geq N_2$ and for any increasing finite sequence $(r_i)_{i=1}^{l}$ with $r_l<N_1$:  
		\[\sum_{i=1}^{l-1} |x_n(r_{i+1})-x_n(r_i)|^p \leq \frac{\varepsilon}{3}
		\quad \text{ and } \quad |x_n(r_{l})-x_n(N_1)|^p \leq \varepsilon' , \]
		where $\varepsilon'$ is chosen so that $||x|-|y||^p \geq |x|^p - \frac{\varepsilon}{3}$ for every $x\in [0,M]$ and $y \in [0,\varepsilon']$ (the map $f(x,y)=\abs{x-y}^p$ is Lipschitz on $[0, M]\times [0,M]$, so $\varepsilon'=\min\{\frac{\varepsilon}{3\Lip(f)},M\}$ does the work). Note also that in such a case one also has  
		\begin{align*}
		|x_n(r) - x_n(N_1)|^p &\geq \abs{|x_n(r) - x_n(r_l)| - |x_n(r_l) - x_n(N_1)|}^p \\
		&\geq |x_n(r) - x_n(r_l)|^p - \frac{\varepsilon}{3},     
		\end{align*}
		for any $r\geq N_1$.
		
		Now, let $n \geq N_2$. It follows from the above estimates that we may find a fixed increasing sequence $(q_i)_{i=1}^{k}$ with $q_1\geq N_1$ and 
		\[\sum_{i=1}^{k-1} |x_n(q_{i+1})-x_n(q_i)|^p \geq \|x_n\|^p- \varepsilon.\]
		Note that $p_N<N_1\leq q_1$ by construction so that 
		\begin{multline*}
		(\star) \quad \|x+x_n\|^p \geq \sum_{i=1}^{N-1}|x(p_{i+1}) -x(p_i) + x_n(p_{i+1}) -x_n(p_i)|^p \\
		+ \sum_{i=1}^{k-1}|x(q_{i+1}) -x(q_i) + x_n(q_{i+1}) -x_n(q_i)|^p .
		\end{multline*}
		But
		\begin{multline*}
		\Big( \sum_{i=1}^{N-1}|x(p_{i+1}) -x(p_i) + x_n(p_{i+1}) -x_n(p_i)|^p  \Big)^{\frac 1p} \\
		\geq \Big( \sum_{i=1}^{N-1} |x(p_{i+1}) -x(p_i)|)^p \Big)^{\frac 1p} - \ep 
		= 1 - \ep,  
		\end{multline*}
		which implies that 
		$\sum_{i=1}^{N-1}|x(p_{i+1}) -x(p_i) + x_n(p_{i+1}) -x_n(p_i)|^p \geq 1 - p\ep.$ \\
		Also, since $x(n)=0$ for $n>N_1$,  
		\[\sum_{i=1}^{N-1}|x(q_{i+1}) -x(q_i) + x_n(q_{i+1}) -x_n(q_i)|^p \geq \|x_n\|^p - \varepsilon.\]
		Finally, we obtain from $(\star)$ the following last estimate
		\[ \|x+x_n\|^p \geq 1 + \|x_n\|^p-(1+p)\ep.\]
	\end{proof}

	As an application of Theorem~\ref{TheoremWeakStarMP} and the two previous lemmata, we obtain the following corollary. 
	
	\begin{corollary} \label{CorJamesWMP}
		If $1<p\leq q< \infty$ then $(\mathcal J_p, \widetilde{\mathcal J_q})$ has the weak$^*$-to-weak$^*$MP.
	\end{corollary}
	
	We now turn to the reverse situation:
	
	\begin{corollary} \label{CorJamesWMP2}
		If $1< q< p <\infty$ then every bounded operator from $\mathcal J_p$ to $\J_q$ is compact.
		In particular, the pair $(\mathcal J_p,\mathcal J_q)$ has the weak$^*$-to-weak$^*$MP.
	\end{corollary}
	\begin{proof}
		Since $\p_{\widetilde{\J_p}}(t) \leq (1 +t^p)^{1/p} - 1$, $\d_{\J_q}^*(t) \geq (1 +t^q)^{1/q} - 1$ and $\d_{\J_q}(t) \geq \d_{\J_q}^*(t)$, we obtain from Lemma~\ref{CompactOperator} that every bounded operator from $\widetilde{\J_p}$ to $\J_q$ is compact. Since $\widetilde{\J_p}$ and $\J_p$ are isomorphic, this yields the same conclusion for every bounded operator from $\J_p$ to $\J_q$. To conclude, Theorem~\ref{TheoremCompactWMP}~(2) provides the fact that the pair $(\mathcal J_p,\mathcal J_q)$ has the weak$^*$-to-weak$^*$MP.
	\end{proof}
	
	\begin{remark}
		Quite surprisingly, the pair $(\mathcal J_2, \R)$ fails the W$^*$MP (but has the weak$^*$-to-normMP thanks to Corollary~\ref{CorJamesWMP}). Indeed, consider $T \colon \J_2 \to \R$ be the linear operator given by 
		\[ \forall x=\big(x(j)\big)_{j=1}^{\infty} \in \J_2, \quad Tx := \frac{-x(1)}{2} + \sum_{j=2}^{\infty} \frac{x(j)}{j^2}.\]
		That is, $T=(\frac{-1}{2},\frac{1}{2^2}, \frac{1}{3^2}, \ldots ) \in \J_2^*$.
		Note that $T$ is not weak$^*$ continuous since $T(s_n) \not\to 0$ while the summing basis $(s_n)$ is weak$^*$-null. Since the standard unit vector basis $(e_n)_{n=1}^{\infty}$ is monotone and shrinking, it follows that for every $x^{**} \in \J_2^{**}$, $\|x^{**}\|_{\J_2^{**}} = \sup_n \|(x^{**}(e_1^*),x^{**}(e_2^*), \ldots,x^{**}(e_n^*), 0 ,\ldots )\|_{\J_2}$ (see e.g. Proposition 4.14 in \cite{FHHMZ11}).
		\smallskip
		
		We claim that any norming functional $x^{**}$ for $T$ is of the form 
		$$x^{**} = \frac{1}{\sqrt{(t_2-t_1)^2+t_2^2}}(t_1,t_2,t_2,\ldots).$$ 
		Indeed, fix any $x^{**}:=(c_n)_n \in S_{\J_2^{**}}$ such that $\langle x^{**} , T \rangle = \|T\|$, and let $s : = \sup_{n \geq 2} |c_n| $. Notice that 
		$$(\star) \quad  1 = \|x^{**}\| = \sup_n \|(c_1 , \ldots , c_n , 0 , \ldots)\|_{\J_2} \geq \sqrt{(s-c_1)^2 + s^2}.$$ 
		Then some easy computations show that
		$$\|T\|=\langle x^{**} , T \rangle 
		= \frac{-c_1}{2} + \sum_{j=2}^{\infty} \frac{c_j}{j^2} 
		\leq \frac{-c_1}{2} + \sum_{j=2}^{\infty} \frac{s}{j^2}.
		$$
		Now the right-hand side of this inequality is equal to $\langle z , T\rangle$, where $z := (c_1,s,s,\ldots)$ belongs to $B_{\J_2^{**}}$ thanks to $(\star)$. Therefore $\langle z , T\rangle \geq \|T\|$ and so $z$ is a norming functional of the desired form. Moreover, if $x^{**} \neq z$ then the last inequality above would be strict and so we would get $\|T\| < \langle z , T\rangle \leq \|T\|$ as a contradiction.
		\smallskip
		
		Consequently, in order to maximize $T(x^{**})$ it suffices to maximize the function 
		\begin{eqnarray*}
			g(t_1,t_2) = \frac{\frac{-t_1}{2} + t_2(\frac{\pi^2}{6}-1)}{\sqrt{(t_2-t_1)^2+t_2^2}} = \frac{\frac{-1}{2} \frac{t_1}{t_2} + (\frac{\pi^2}{6}-1)}{\sqrt{(1-\frac{t_1}{t_2})^2+(\frac{t_1}{t_2})^2}}.
		\end{eqnarray*}
		We set $t = \frac{t_1}{t_2}$ and we are now looking for the maximum of the following  map:
		\[ f(t) =  \frac{\frac{-1}{2}t + (\frac{\pi^2}{6}-1)}{\sqrt{(1-t)^2+t^2}}.\]
		A basic study of the map $f$ shows that it attains its maximum at $t_{max} =  \frac{\pi^2-9}{2\pi^2-15}$ (and $f(t_{max}) \simeq 0.66$). To conclude, fix $t_1,t_2 \in \R$ such that $\frac{t_1}{t_2} = t_{max}$ and $\sqrt{(t_2-t_1)^2+t_2^2}=1$. Now for every $n \in \N$ let $x_n \in \J_2$ be such that $x_n(1)=t_1$, $x_n(2)= \ldots = x_n(n) = t_2$ and $x_n(j) = 0$ whenever $j>n$ (that is, $x_n = (t_1, t_2 ,\ldots ,t_2 , 0 ,\ldots)$). Then $(x_n)_{n=1}^\infty$ is a normalized maximizing sequence for $T$ which is non-weak$^*$ null. However, $T$ does not attain its norm on $\J_2$. Indeed, we proved that any norming functional for $T$ must be of the form $x=(t_1,t_2,t_2, \ldots)$. Such a functional $x$ belongs to $\J_2$ if and only if $t_2 = 0$ (as every element in $\J_2$ is a sequence that must converge to 0). So one has $|\langle x , T \rangle| = |\frac{-t_1}{2}| = \frac{|t_1|}{2} = \frac{\|x\|}{2} \leq \frac{1}{2} < f(t_{\max} )$ and so a norming functional for $T$ cannot belong to $J_2$.
	\end{remark}

	Having in mind the previous corollaries, it is quite natural to wonder the following:
	
	\begin{question}
		Does the pair $(\mathcal J_p, \mathcal J_q)$ has the weak$^*$-to-weak$^*$MP for $1<p\leq q< \infty$?
	\end{question}

	\subsection{Orlicz spaces} 
	
	Given an Orlicz function $\varphi\colon [0,+\infty)\to[0,+\infty)$ (that is, $\varphi$ is a continuous convex unbounded function with $\varphi(0)=0$), the Orlicz sequence space $\ell_\varphi$ is the space of all real sequences $x=(x_n)_{n=1}^\infty$ such that $\sum_{n=1}^\infty \varphi(|x_n|/\lambda)<\infty$. It is a Banach space when equipped with the Luxemburg norm:
	\[ \norm{x}_{\varphi}=\inf\{\lambda>0 : \sum_{n=1}^\infty \varphi(|x_n|/\lambda)\leq 1\}.\]
	The closed linear span of $\{e_n : n\in \mathbb N\}$ in $\ell_{\varphi}$ is denoted $h_{\varphi}$. The space $h_\varphi$ coincides with $\ell_\varphi$ precisely if $\varphi$ satisfies the $\Delta_2$ condition at zero, i.e. $\limsup_{t\to 0}\varphi(2t)/\varphi(t)<\infty$. The space $h_\varphi$ (or $\ell_\varphi$) is reflexive if and only if both $\varphi$ and $\varphi^*$
	satisfy the $\Delta_2$ condition at zero, where $\varphi^*(t)=\sup\{st-\varphi(s)\}$ is the convex conjugate of $\varphi$. The Boyd indices of an Orlicz function $\varphi$ are defined as follows:
	\begin{align*} \alpha_\varphi &= \sup\{p>0 : \sup_{0<u,t\leq 1} \frac{\varphi(tu)}{\varphi(u)t^p}<\infty\}, \\ \beta_\varphi &= \inf\{p>0 : \inf_{0<u,t\leq 1} \frac{\varphi(tu)}{\varphi(u)t^p}>0\}.
	\end{align*}
	It is known that $\beta_\varphi<\infty$ precisely if $\varphi$ satisfies the $\Delta_2$ condition at $0$. The asymptotic moduli of the space $h_\varphi$ is linked to the Boyd indices: $h_\varphi$ is AUS (resp. AUC) if and only if $\alpha_\varphi>1$ (resp. $\beta_\varphi<\infty$). Moreover $\alpha_\varphi$ is the supremum of the numbers $\alpha$ such that $\p_{h_\varphi}$ has power type $\alpha$ \cite{GJT07}. In addition,  $\beta_\varphi$ is the infimum of the numbers $\beta$ such that $\d_{h_\varphi}$ has power type $\beta$ \cite{BM10}. 
	\smallskip
	
	In order to apply Theorem~\ref{TheoremWeakStarMP} to the case of Orlicz spaces we need an estimation of $\p_{h_\varphi}(t)$ and $\d_{h_\varphi}(t)$ for all $t>0$. To this end, consider the following indices:
	\begin{align*} p_\varphi &= \sup\{p>0 : u^{-p} \varphi(u) \text{ is non-decreasing for all } 0<u\leq\varphi^{-1}(1)\} \\
	q_\varphi &= \inf\{p>0 : u^{-p} \varphi(u) \text{ is non-increasing for all } 0<u\leq\varphi^{-1}(1)\}.
	\end{align*}
	
	Clearly $1\leq p_{\varphi}\leq q_{\varphi}\leq \infty$. Moreover, $\varphi$ (resp. $\varphi^*$) satisfies $\Delta_2$ condition at $0$ if and only if $q_\varphi<\infty$ (resp. $p_\varphi>1$); see  \cite{Maligranda85} or \cite{Delpech09}. Delpech \cite{Delpech09} showed that if $q_\varphi<\infty$ then 
	\[ (1+t^{q_\varphi})^{1/q_\varphi}-1\leq \d_{h_\varphi}(t)\]
	for all $t\in (0,1]$, but actually the proof works for all $t>0$. Analogously, one can show that  
	\[  \p_{h_\varphi}(t)\leq (1+t^{p_\varphi})^{1/p_\varphi}-1\]
	for all $t>0$.
	
	\begin{remark} In general $[\alpha_\varphi, \beta_\varphi]\subset [p_\varphi, q_\varphi]$, but the inclusion may be strict \cite{Maligranda85}. Thus the previous inequalities are not tight. Also, recall that, given any $p\in [\alpha_\varphi, \beta_\varphi]$, the space $h_\varphi$ contains almost isometric copies of $\ell_p$. Thus $\d_{\ell_\varphi}(t)\leq (1+t^{\alpha_\varphi})^{1/\alpha_\varphi}-1$ and $\p_{\ell_\varphi}(t)\geq (1+t^{\beta_\varphi})^{1/\beta_\varphi}-1$ for all $t>0$. 
	\end{remark}

	\begin{corollary} Let $\varphi, \psi$ be Orlicz functions. Assume that $h_\varphi, h_\psi$ are reflexive and $q_{\varphi}\leq p_{\psi}$. Then the pair $(h_\varphi, h_\psi)$ has the WMP.  
	\end{corollary}
	
	\begin{remark}
		Given Orlicz functions $\varphi, \psi$ such that that $h_\varphi, h_\psi$ are reflexive, if  $\alpha_\varphi>\beta_\psi$ then every operator from $h_\varphi$ to $h_\psi=\ell_\psi$ is compact \cite{AO97}, and so $(h_\varphi, h_\psi)$ has the WMP.
	\end{remark}

	\subsection{Remarks about the pair \texorpdfstring{$(L_p,L_q)$}{(Lp,Lq)}}
	In what follows, $L_p$ stands for $L_p([0,1])$ with $1<p<\infty$. It is known (see \cite{Milman71} at page 117 for instance) that if $X = L_p$ then there exist constants $C_1(p),C_2(p)$ such that for $1<p<2$ we have
	\begin{eqnarray*}
		C_1(p) t^2  \leq &\d_X(t)&  \leq (p-1) t^2 \\
		\frac 1p  t^p  \leq  &\p_X(t)& \leq \frac 2p t^p 
	\end{eqnarray*}
	and for $2<p<\infty$ we have that 
	\begin{eqnarray*}
		C_1(p) t^p  \leq &\d_X(t)&  \leq \frac 1p t^p \\
		(p-1) t^2    \leq  &\p_X(t)& \leq C_2(p) t^2. 
	\end{eqnarray*}
	Consequently, for $1<p<q<\infty$, we cannot apply Theorem~\ref{TheoremWeakStarMP} to prove that the pair $(L_p,L_q)$ has the WMP. Nevertheless, thanks to Theorem \ref{thm:renorming} and the above estimations we can still say something for a particular renorming of $L_p$ and $L_q$.

	\begin{corollary}
		Assume that $1<p\leq2$ and $2\leq q<\infty$. Then there are equivalent norms $\eqnorma_p$ on $L_p$ and $\eqnorma_q$ on $L_q$ such that the pair $((L_p, \eqnorma_p), (L_q, \eqnorma_q))$ has the WMP.
	\end{corollary}

	Replacing the weak topology by the topology $\tau_m$ of convergence in measure, we can apply similar techniques to prove the following related results. 
	\begin{proposition} \label{propConvMeasure}
		Let $1<p<q<\infty$ and let $T \colon L_p \to L_q$ be a bounded operator which is $\tau_m$-to-$\tau_m$ continuous. If there exists a maximizing sequence $(x_n)_n \subset L_p$ for $T$ which $\tau_m$-converges to some $x \neq 0$, then $T$ attains its norm at $x$. 
	\end{proposition}
	\begin{proof}  
		The proof follows the same lines as in Theorem~\ref{TheoremWeakStarMP}. In fact, $L_p$ has the following property (see \cite{KW} e.g.): if $(x_n)_n \subset L_p$ converges to $0$ in measure, then for every $x \in L_p$
		\[(\star) \qquad \limsup_n \|x+x_n \| = \big(\|x\|^p+\limsup_n \|x_n\|^p \big)^{1/p}.\]
		Let $T \colon L_p \to L_q$ be a bounded operator which is $\tau_m$-to-$\tau_m$ continuous and let $(x_n)_n \subset L_p$ be a maximizing sequence for $T$ which $\tau_m$-converges to some $x \neq 0$. Without loss of generality, we may assume that $T$ has norm 1.
		Since $T$ is $\tau_m$-to-$\tau_m$ continuous, we have that $Tx_n \overset{\tau_m}{\underset{n}{\longrightarrow}} Tx$. Using $(\star)$, we thus have the following estimates: 
		\begin{eqnarray*}
			1 &=& \norm T = \lim\limits_{n} \norm{Tx_n}  = \lim\limits_{n} \norm{Tx+Tx_n - Tx} \\
			&=& \big(\|Tx\|^q +\lim\limits_{n} \|Tx_n - Tx\|^q \big)^{\frac{1}{q}} \\
			&\leq& \big(\|Tx\|^q +\lim\limits_{n} \|x_n - x\|^q \big)^{\frac{1}{q}} \\
			&\leq& \big(\|Tx\|^p +\lim\limits_{n} \|x_n - x\|^p \big)^{\frac{1}{p}} \\
			&=& \big( \|Tx\|^p + 1-\|x\|^p  \big)^{\frac{1}{p}}
		\end{eqnarray*}
		Therefore, we deduce that $\|Tx\| \geq \|x\|$, which finishes the proof. 
	\end{proof}
	
	The lack of $\tau_m$-compactness of the unit ball of $L_p$ forces us to consider pairs $(X,L_q)$ where $X$ is a subspace of $L_p$ whose unit ball $B_X$ is $\tau_m$-compact. Let us point out that, given $X\subset L_p$, $1<p<\infty$, the following statements are equivalent:
	\begin{enumerate}[i)]
		\item $B_X$ is $\tau_m$-compact;
		\item $X$ embeds almost isometrically into $\ell_p$;
		\item $B_X$ is compact for the topology inherited by the $L_1$-norm.
	\end{enumerate}
	Indeed, it is clear that iii)$\Rightarrow$ i). i)$\Rightarrow$ ii)	follows from the following well-known facts for a finite measure space $(\Omega,\Sigma,\mu)$:
	\begin{itemize}
		\item If $H \subset L_p(\mu)$ is bounded, then it is uniformly integrable as a subset of $L_1(\mu)$ (this is a well-known, easy application of Holder’s inequality).
		\item If $H \subset L_1(\mu)$ is uniformly integrable and relatively compact in measure, then it is relatively norm compact in $L_1(\mu)$. Indeed, this is an immediate application of Vitali’s convergence theorem.
	\end{itemize} Finally, it is proved in \cite{KW} that ii) and iii) are equivalent. 
	
	\begin{corollary} 
		Let $1<p<q<\infty$. If $X$ is a subspace of $L_p$ such that $B_X$ is $\tau_m$-compact, then the pair $(X,L_q)$ has the $\tau_m$-to-$\tau_m$MP.
	\end{corollary}
	
	In the case of almost everywhere convergence, a related result is obtained from the celebrated Br\'ezis-Lieb lemma, which says that if $(f_n)\subset L_p$ is bounded and converges a.e. to $f$, then $\lim_n\norm{f_n}^p_p-\norm{f-f_n}_p^p=\norm{f}_p^p$. As a consequence, if $1\leq p\leq q<\infty$ and $T\in \mathcal L(L_p,L_q)$ admits a maximizing sequence $(f_n)$ with $f_n\to f$ a.e., $f\neq 0$, and $Tf_n\to Tf$ a.e., then $T$ attains its norm at $f$ (see \cite{BL}).
	
	Let us briefly discuss the case when $1<q<p<\infty$. For the pair $(\ell_p,\ell_q)$, this was handled by Pitt’s Theorem (or more generally by Theorem~\ref{TheoremCompactWMP}) saying that every operator from $\ell_p$ to $\ell_q$ is compact. This approach fails for $(L_p,L_q)$. Indeed, 
	H.~P.~Rosenthal characterized in \cite[Theorem~A.2]{Rosenthal69} when every operator from $L_p(\mu)$ to $L_q(\nu)$ is compact. That is never the case for $L_p[0,1]$ and $L_q[0,1]$. However, if $q<p$, $1\leq q<2$, then every operator from $L_p[0,1]$ to $\ell_q$ is compact, and if $q<p$ and $2<p$ then every operator from $\ell_p$ to $L_q[0,1]$ is compact.

	\section*{Acknowledgments}
	
	Part of this work was carried out during a visit of the two named authors in Murcia (Spain) in January 2020. 
	They are deeply grateful to Mat\'ias Raja and the ``Universidad de Murcia" for the hospitality and excellent working conditions there.
	The authors would like to thank Richard M. Aron, Jos\'e Rodr\'iguez and Abraham Rueda Zoca for very useful conversations, and to the referee for the careful reading of the paper.
	The first author is supported in part by the grants MTM2017-83262-C2-2-P and Fundaci\'on S\'eneca Regi\'on de Murcia 20906/PI/18 also supported by a postdoctoral grant from Fundaci\'on S\'eneca.

	\bibliographystyle{siam}
	\bibliography{aus.bib}

\begin{thebibliography}{10}

\bibitem{albiackalton}
{\sc F.~Albiac and N.~J. Kalton}, {\em Topics in {B}anach space theory},
  vol.~233 of Graduate Texts in Mathematics, Springer, [Cham], second~ed.,
  2016.
\newblock With a foreword by Gilles Godefroy.

\bibitem{AG19}
{\sc R.~J. Aliaga and A.~J. Guirao}, {\em On the preserved extremal structure
  of {L}ipschitz-free spaces}, Studia Math., 245 (2019), pp.~1--14.

\bibitem{AP20}
{\sc R.~J. Aliaga, E.~Perneck\'{a}, C.~Petitjean, and A.~Proch\'{a}zka}, {\em
  Supports in {L}ipschitz-free spaces and applications to extremal structure},
  J. Math. Anal. Appl., 489 (2020), pp.~124128, 14.

\bibitem{AliprantisBorder}
{\sc C.~D. Aliprantis and K.~C. Border}, {\em Infinite dimensional analysis},
  Springer, Berlin, third~ed., 2006.
\newblock A hitchhiker's guide.

\bibitem{Aron_WMP19}
{\sc R.~M. Aron, D.~Garc\'{\i}a, D.~Pellegrino, and E.~V. Teixeira}, {\em
  Reflexivity and nonweakly null maximizing sequences}, Proc. Amer. Math. Soc.,
  148 (2020), pp.~741--750.

\bibitem{AO97}
{\sc E.~A. Ausekle and {\`E}.~F. Oya}, {\em Pitt's theorem for {L}orentz and
  {O}rlicz sequence spaces}, Mat. Zametki, 61 (1997), pp.~18--25.

\bibitem{BM10}
{\sc L.~Borel-Mathurin}, {\em Isomorphismes non lin\'eaires entre espaces de
  {B}anach}, PhD thesis, Universit\'e Paris 6, 2010.

\bibitem{BL}
{\sc H.~Br\'{e}zis and E.~Lieb}, {\em A relation between pointwise convergence
  of functions and convergence of functionals}, Proc. Amer. Math. Soc., 88
  (1983), pp.~486--490.

\bibitem{Castillo}
{\sc J.~M.~F. Castillo and M.~Gonz\'{a}les}, {\em On the {D}unford-{P}ettis
  property in {B}anach spaces}, Acta Univ. Carolin. Math. Phys., 35 (1994),
  pp.~5--12.

\bibitem{CauseyLancien}
{\sc R.~M. Causey and G.~Lancien}, {\em Prescribed {S}zlenk index of separable
  {B}anach spaces}, Studia Math., 248 (2019), pp.~109--127.

\bibitem{Delpech09}
{\sc S.~Delpech}, {\em Asymptotic uniform moduli and {K}ottman constant of
  {O}rlicz sequence spaces}, Rev. Mat. Complut., 22 (2009), pp.~455--467.

\bibitem{DiestelDP}
{\sc J.~Diestel}, {\em A survey of results related to the {D}unford-{P}ettis
  property}, in Proceedings of the {C}onference on {I}ntegration, {T}opology,
  and {G}eometry in {L}inear {S}paces ({U}niv. {N}orth {C}arolina, {C}hapel
  {H}ill, {N}.{C}., 1979), vol.~2 of Contemp. Math., Amer. Math. Soc.,
  Providence, R.I., 1980, pp.~15--60.

\bibitem{DKLR}
{\sc S.~J. Dilworth, D.~Kutzarova, G.~Lancien, and N.~L. Randrianarivony}, {\em
  Equivalent norms with the property {$(\beta)$} of {R}olewicz}, Rev. R. Acad.
  Cienc. Exactas F\'{\i}s. Nat. Ser. A Mat. RACSAM, 111 (2017), pp.~101--113.

\bibitem{FHHMZ11}
{\sc M.~Fabian, P.~Habala, P.~H{\'a}jek, V.~Montesinos, and V.~Zizler}, {\em
  Banach space theory}, CMS Books in Mathematics/Ouvrages de Math\'ematiques de
  la SMC, Springer, New York, 2011.
\newblock The basis for linear and nonlinear analysis.

\bibitem{GPPR_2018}
{\sc L.~Garc\'{\i}a-Lirola, C.~Petitjean, A.~Proch\'{a}zka, and A.~Rueda~Zoca},
  {\em Extremal structure and duality of {L}ipschitz free spaces}, Mediterr. J.
  Math., 15 (2018), pp.~Art. 69, 23.

\bibitem{Girardi}
{\sc M.~Girardi}, {\em The dual of the {J}ames tree space is asymptotically
  uniformly convex}, Studia Math., 147 (2001), pp.~119--130.

\bibitem{GKL00}
{\sc G.~Godefroy, N.~Kalton, and G.~Lancien}, {\em Subspaces of {$c_0(\mathbf
  N)$} and {L}ipschitz isomorphisms}, Geom. Funct. Anal., 10 (2000),
  pp.~798--820.

\bibitem{GJT07}
{\sc R.~Gonzalo, J.~A. Jaramillo, and S.~L. Troyanski}, {\em High order
  smoothness and asymptotic structure in {B}anach spaces}, J. Convex Anal., 14
  (2007), pp.~249--269.

\bibitem{Biorthogonal}
{\sc P.~H\'{a}jek, V.~Montesinos~Santaluc\'{\i}a, J.~Vanderwerff, and
  V.~Zizler}, {\em Biorthogonal systems in {B}anach spaces}, vol.~26 of CMS
  Books in Mathematics/Ouvrages de Math\'{e}matiques de la SMC, Springer, New
  York, 2008.

\bibitem{Jaramillo2000}
{\sc J.~A. Jaramillo, A.~Prieto, and I.~Zalduendo}, {\em Sequential
  convergences and {D}unford-{P}ettis properties}, Ann. Acad. Sci. Fenn. Math.,
  25 (2000), pp.~467--475.

\bibitem{Johnson77}
{\sc W.~B. Johnson}, {\em On quotients of {$L_{p}$} which are quotients of
  {$l_{p}$}}, Compositio Math., 34 (1977), pp.~69--89.

\bibitem{JLPS02}
{\sc W.~B. Johnson, J.~Lindenstrauss, D.~Preiss, and G.~Schechtman}, {\em
  Almost {F}r\'echet differentiability of {L}ipschitz mappings between
  infinite-dimensional {B}anach spaces}, Proc. London Math. Soc. (3), 84
  (2002), pp.~711--746.

\bibitem{KaltonCompact}
{\sc N.~J. Kalton}, {\em {$M$}-ideals of compact operators}, Illinois J. Math.,
  37 (1993), pp.~147--169.

\bibitem{Kalton04}
\leavevmode\vrule height 2pt depth -1.6pt width 23pt, {\em Spaces of
  {L}ipschitz and {H}\"older functions and their applications}, Collect. Math.,
  55 (2004), pp.~171--217.

\bibitem{KaltonTAMS2013}
\leavevmode\vrule height 2pt depth -1.6pt width 23pt, {\em Uniform
  homeomorphisms of {B}anach spaces and asymptotic structure}, Trans. Amer.
  Math. Soc., 365 (2013), pp.~1051--1079.

\bibitem{KW}
{\sc N.~J. Kalton and D.~Werner}, {\em Property {$(M)$}, {$M$}-ideals, and
  almost isometric structure of {B}anach spaces}, J. Reine Angew. Math., 461
  (1995), pp.~137--178.

\bibitem{KOS99}
{\sc H.~Knaust, E.~Odell, and T.~Schlumprecht}, {\em On asymptotic structure,
  the {S}zlenk index and {UKK} properties in {B}anach spaces}, Positivity, 3
  (1999), pp.~173--199.

\bibitem{Kover2005}
{\sc J.~Kover}, {\em Compact perturbations and norm attaining operators},
  Quaest. Math., 28 (2005), pp.~401--408.

\bibitem{LancienSurvey}
{\sc G.~Lancien}, {\em A survey on the {S}zlenk index and some of its
  applications}, RACSAM. Rev. R. Acad. Cienc. Exactas F\'{\i}s. Nat. Ser. A
  Mat., 100 (2006), pp.~209--235.

\bibitem{LT77}
{\sc J.~Lindenstrauss and L.~Tzafriri}, {\em Classical {B}anach spaces. {I}},
  Springer-Verlag, Berlin-New York, 1977.
\newblock Sequence spaces, Ergebnisse der Mathematik und ihrer Grenzgebiete,
  Vol. 92.

\bibitem{Maligranda85}
{\sc L.~Maligranda}, {\em Indices and interpolation}, Dissertationes Math.
  (Rozprawy Mat.), 234 (1985), p.~49.

\bibitem{Milman71}
{\sc V.~D. Milman}, {\em Geometric theory of {B}anach spaces. {II}. {G}eometry
  of the unit ball}, Uspehi Mat. Nauk, 26 (1971), pp.~73--149.

\bibitem{Netillard}
{\sc F.~Netillard}, {\em Coarse {L}ipschitz embeddings of {J}ames spaces},
  Bull. Belg. Math. Soc. Simon Stevin, 25 (2018), pp.~71--84.

\bibitem{Odell2007}
{\sc E.~Odell, T.~Schlumprecht, and A.~Zs\'{a}k}, {\em Banach spaces of bounded
  {S}zlenk index}, Studia Math., 183 (2007), pp.~63--97.

\bibitem{Pellegrino_09}
{\sc D.~Pellegrino and E.~V. Teixeira}, {\em Norm optimization problem for
  linear operators in classical {B}anach spaces}, Bull. Braz. Math. Soc.
  (N.S.), 40 (2009), pp.~417--431.

\bibitem{Prus87}
{\sc S.~Prus}, {\em Finite-dimensional decompositions of {B}anach spaces with
  {$(p,q)$}-estimates}, Dissertationes Math. (Rozprawy Mat.), 263 (1987),
  p.~45.

\bibitem{Raja2013}
{\sc M.~Raja}, {\em On asymptotically uniformly smooth {B}anach spaces}, J.
  Funct. Anal., 264 (2013), pp.~479--492.

\bibitem{Rosenthal69}
{\sc H.~P. Rosenthal}, {\em On quasi-complemented subspaces of {B}anach spaces,
  with an appendix on compactness of operators from {$L^{p}\,(\mu )$} to
  {$L^{r}\,(\nu )$}}, J. Functional Analysis, 4 (1969), pp.~176--214.

\end{thebibliography}

\end{document}